\documentclass[11pt]{amsart} 
\usepackage{latexsym}
\usepackage{amsfonts}
\usepackage{amsmath,amsthm,amssymb}
\usepackage{fullpage}
\usepackage{mathabx}
\usepackage{tikz}
\usetikzlibrary{matrix}
\usepackage[colorlinks=true, linkcolor=blue, citecolor=magenta]{hyperref}
\usepackage{mathrsfs}
\usepackage{soul}
\usepackage{arydshln}

\newcommand{\cal}{\mathcal}

\def\epsilon{\varepsilon}
\def\phi{\varphi}

\def\subset{\subseteq}

\newcommand{\Out}{\mbox{Out}}

\newcommand{\card}{\mbox{card}}

\newcommand{\bvec}{\overleftarrow}

\newcommand{\FN}{F_N}   % F ou F_n ou F_N ?

\newcommand{\R}{\mathbb R}
\newcommand{\Z}{\mathbb Z}

\newcommand{\N}{\mathbb N}

\def\strutdepth{\dp\strutbox}
\def \ss{\strut\vadjust{\kern-\strutdepth \sss}}
\def \sss{\vtop to \strutdepth{
\baselineskip\strutdepth\vss\llap{$\diamondsuit\;\;$}\null}}

\def\strutdepth{\dp\strutbox}
\def \sst{\strut\vadjust{\kern-\strutdepth \ssss}}
\def \ssss{\vtop to \strutdepth{
\baselineskip\strutdepth\vss\llap{$\spadesuit\;\;$}\null}}

\def\strutdepth{\dp\strutbox}
\def \ssh{\strut\vadjust{\kern-\strutdepth \sssh}}
\def \sssh{\vtop to \strutdepth{
\baselineskip\strutdepth\vss\llap{$\heartsuit\;\;$}\null}}

\def\qed{\hfill\rlap{$\sqcup$}$\sqcap$\par}

\def\strutdepth{\dp\strutbox}
\def \ss{\strut\vadjust{\kern-\strutdepth \sss}}
\def \sss{\vtop to \strutdepth{
\baselineskip\strutdepth\vss\llap{$\diamondsuit\;\;$}\null}}

\def\strutdepth{\dp\strutbox}
\def \sst{\strut\vadjust{\kern-\strutdepth \ssss}}
\def \ssss{\vtop to \strutdepth{
\baselineskip\strutdepth\vss\llap{$\spadesuit\;\;$}\null}}
\def\qed{\hfill\rlap{$\sqcup$}$\sqcap$\par}

\newtheorem{thm}{Theorem}[section]
\newtheorem{cor}[thm]{Corollary}
\newtheorem{lem}[thm]{Lemma}
\newtheorem{prop}[thm]{Proposition}

\newtheorem{fact}[thm]{Fact}
\theoremstyle{definition}
\newtheorem{defn}[thm]{Definition}
\newtheorem{example}[thm]{Example}

\newtheorem{rem}[thm]{Remark}
\newtheorem{defn-rem}[thm]{Definition-Remark}

\newtheorem{question}[thm]{Question}
\theoremstyle{remark}

\numberwithin{equation}{section}

\begin{document}

\author[N.~Bedaride]{Nicolas B\'edaride}
\author[A.~Hilion]{Arnaud Hilion}
\author[M.~Lustig]{Martin Lustig}

\address{\tt 
Aix Marseille Universit\'e, CNRS, Centrale Marseille, I2M UMR 7373,
13453  Marseille, France}
\email{\tt Nicolas.Bedaride@univ-amu.fr}
\email{\tt Arnaud.Hilion@univ-amu.fr}
\email{\tt Martin.Lustig@univ-amu.fr}

\title[Tower power]{Tower power for $S$-adics}
 
\begin{abstract} 
We explain and restate the results from our recent paper \cite{BHL1} in standard language for substitutions and $S$-adic systems in symbolic dynamics. We then produce as rather direct 
application an $S$-adic system (with finite set of substitutions $S$ on $d$ letters) that is minimal and 
has $d$ distinct ergodic probability measures. 

As second application we exhibit a formula that allows an efficient practical computation of the cylinder measure $\mu([w])$, for any word $w \in \cal A^*$ and any invariant measure $\mu$ on the subshift $X_\sigma$ defined by any everywhere growing but not necessarily primitive or irreducible substitution $\sigma: \cal A^* \to \cal A^*$.
Several examples are considered in detail, and model computations are presented.
\end{abstract}

%{Primary 20F36, Secondary 20E36, 57M05} 
%\subjclass[2000]{Primary 20F, Secondary 20E, 57M}
%, 37B, 37D}
 
 
\maketitle 

\section{Introduction}
\label{introduction} 

Symbolic dynamics has undergone in the past 20-30 years a sequence of several fundamental changes in what people considered the most promising tool to investigate symbolic dynamical systems and their invariant measures. After first having developed the technology of Kakutani-Rokhlin towers, Bratteli diagrams with Vershik maps were recognized as (in many situations) more promising. More recently $S$-adic systems have moved into the lime light.

The authors of this note have very recently made public an 
exposition \cite{BHL1} of 
yet another technology, which we believe 
to be perhaps easier to learn and 
in some situations 
more efficient to apply than the previously existing ones. It also seems to have in many situations the potential for further reaching results.

While writing the paper \cite{BHL1} the authors also had in mind future applications to geometric group theory, most notably to currents on free groups $\FN$ and to the action of the automorphism group $\Out(\FN)$ on current space. Since in this world it is natural to admit at all times inverses to the letters that generate the system, the paper \cite{BHL1} was written for graphs and graph maps rather than for monoids and substitutions, which makes our results less accessible for 
somebody working in the traditional 
settings 
of symbolic dynamics.

Hence a translation into the classical terminology of symbolic dynamics seemed to be called for, and it is provided here (see section \ref{from-BHL1}). Instead of a proof (which is a direct consequence of the main result of \cite{BHL1}, together with the translation dictionary given in section 3 of \cite{BHL1}), we proceed here to present two applications of our new technology. Both are new results (to our knowledge), 
but what is almost more important to the authors is to convince the reader about the relatively small effort by which they are derived from the results of \cite{BHL1}: We thus would like to point the symbolic dynamics community's attention to this new technology, which we believe might be useful also for other people working in this area.

\medskip

We now give a brief description of the results of this paper:

\smallskip

Let $\cal A$ be a finite alphabet of cardinality $d \geq 1$, and let $X$ be a subshift over $\cal A$. Assume that $X$ possesses an $S$-adic expansion over $\cal A$ (see section \ref{prelims}), where $S$ is a possibly infinite set of  substitutions $\sigma_n: \cal A^* \to \cal A^*$, then it is well known that $X$ supports at most $d$ distinct ergodic probability measures. Our first result, 
presented in section \ref{d-measures1}, concerns the realization of this bound and can be stated as follows:

\begin{prop}
\label{d-measures}
(1)
For any integer $d \geq 1$ there exists a directive sequence $\sigma = \sigma_0 \circ \sigma_1 \circ \ldots$, with level alphabets $\cal A_n$ all of cardinality $d$, such that the associated subshift $X_\sigma$ is minimal and supports $d$ distinct invariant ergodic probability measures.

\smallskip
\noindent
(2) 
There is a finite set $S$ (consisting of $4$ substitutions) such that the above substitutions $\sigma_0, \sigma_1, \ldots$ can all be chosen from $S$.
\end{prop}

This realization result contrasts in part (1) 
with the well-known upper bound $\lfloor \frac{d}{2} \rfloor$ for the number of ergodic probability measures (see 
\cite{Katok, Veech}), in the case where $X_\sigma$ is read off from an 
interval exchange transformation. 
Part (2) seems to contradict at first glance known results from Bratteli-Vershik theory, see Remark \ref{comment-referee} below.
We would also like to mention that Proposition \ref{d-measures} is preceded in the literature by several related results, see for instance \cite{ABKK}, \cite{1709.00055} and \cite{FFT}. 

\medskip

Our second result concerns the concrete calculation of the measure of any cylinder $[w]$ with $w \in \cal A^*$, for an arbitrary subshift $X$ over $\cal A = \{a_1, \ldots, a_d\}$. If $X$ is uniquely ergodic, the question is unambiguous, but in general one needs to specify the measure $\mu$ in question. The main result of \cite{BHL1} as stated here in Theorem \ref{thm1} presents a tool for such a specification in rather practical terms, so that a direct calculation (with controlled error term) is possible, once an everywhere growing $S$-adic expansion of $X$ is chosen (see Corollary \ref{cylinder-computation}).

In section \ref{cylinder-determination} we concentrate on the special case of a substitution subshift $X = X_\sigma$ for any everywhere growing (but not necessarily primitive or irreducible) substitution $\sigma: \cal A^* \to \cal A^*$. It is known (see Proposition \ref{mieux-que-moulinette}) that the ergodic measures $\mu_i$ on $X_\sigma$ are in 1-1 correspondence with certain ``distinguished'' eigenvectors $\vec v_i \geq \vec 0$ with eigenvalues $\lambda_i > 1$ of the incidence matrix $M_{\sigma}$. 
Our determination of the cylinder measures is based on a natural extension of the incidence matrix $M_{\sigma}$ to a $(d^2+d)\times(d^2+d)$-matrix $M_\sigma^+$ (for $d = \card \cal A$), and on corresponding prolongations of the eigenvectors $\vec v_i$ to eigenvectors $\vec v^+_i$, as well as on 
``occurrence vectors'' $\vec v_n(w)$ obtained from counting the occurrences of $w$ as factor in the words $\sigma^n(a_i)$ and $\sigma^n(a_i a_j)$.
The precise terms of the following proposition are explained below in section \ref{cylinder-determination};
it should be noted though that both, $\vec v_n(w)$ and the $\vec v_i^+$, can be readily computed from $\sigma$ and $w$, and also the lower bound ``$(\sigma, w)$-large'' used below.

\begin{prop}
\label{cylinder-formula}
Let $\sigma: \cal A^* \to \cal A^*$ be an everywhere growing substitution, and let $\mu = \sum c_i \mu_i$ be any invariant measure on the substitution subshift $X_\sigma$, expressed as non-negative linear combination of 
ergodic measures $\mu_i$.

Then for any $w \in \cal A^*$ and any $(\sigma, w)$-large integer $n \in \N$ the measure of the cylinder $[w]$ is given by the scalar product 
$$\mu([w]) \, =\, 
\langle \, \vec v_n(w), \, \sum c_i \frac{1}{\lambda_i^n} \vec v^+_i \,\rangle$$
\end{prop}

Several examples where the formula from Proposition \ref{cylinder-formula} is  applied to calculate the value of concretely given cylinders are given at the end of 
the paper in 
section \ref{section-examples}.

\medskip
\noindent
{\em Acknowledgements:} The authors would like to thank Julien Cassaigne and Pascal Hubert for useful comments, as well as 
our 
marseillan symbolic dynamics community for its inspiring atmosphere.
We would also like to thank the referee for 
his careful reading of 
the first version, and 
for having 
encouraged us to include Remark \ref{comment-referee}.

%%%%%%%%%%%%%%%%
\section{Preliminaries}
\label{prelims}

In this section we only review standard definitions and facts, and we set up the notation used in this paper.
We use \cite{BD} as standard reference; indeed, we try to use as much as possible their terminology and notations.

\subsection{Subshifts}
\label{generalities}

${}^{}$

Let $\cal A = \{a_1, \ldots, a_d\}$ be a finite set, called {\em alphabet}. We denote by $\cal A^*$ the free monoid over $\cal A$. 
For any element $w \in \cal A^*$ we denote by $|w |$ the length of $w$ as word in the alphabet $\cal A$. For any any two words $v, w \in \cal A^*$ we write $|w|_v$ for the number of occurrences of $v$ as factor of $w$.

We denote by
$$\Sigma_\cal A = \{\ldots x_{-1} x_0 x_1 x_2 \ldots \mid x_i \in \cal A \}$$
be the set of biinfinite words in $\cal A$, called the {\em full shift} over $\cal A$. 
For any word 
$w = w_1 \ldots w_s \in \cal A^*$ 
the {\em cylinder} 
$$[w] \subseteq\Sigma_\cal A$$
is 
the set of all biinfinite words $\ldots x_{-1} x_0 x_1 x_2 \ldots$ in $\cal A$ which satisfy 
$x_1 = w_1, \ldots, x_s = w_s$. The full shift $\Sigma_\cal A$, being in bijection with the set $\cal A^\Z$, is naturally equipped with the product topology, where $\cal A$ is given the discrete topology. For any $w \in \cal A^*$ the cylinder $[w]$
is closed and open. The full shift $\Sigma_\cal A$ is compact, and indeed it is a Cantor set.

The shift map $S:\Sigma_\cal A \to \Sigma_\cal A$ is defined for $x = \ldots x_{-1} x_0 x_1 x_2 \ldots$ by $S(x) = \ldots y_{-1} y_0 y_1 y_2 \ldots$, with $y_n = x_{n+1}$ for all $n \in \Z$. It is bijective and continuous with respect to the above product topology, and hence a homeomorphism. 

A {\em subshift} is a non-empty closed subset $X$ of $\Sigma_\cal A$ which is invariant under the shift map $S$. Such a subshift $X$ is called {\em minimal} if it is the closure of the shift-orbit of any $x \in X$.

Let $\mu$ be a finite Borel measure supported on a subshift $X\subseteq\Sigma_\cal A$. The measure is called {\em invariant} if for every measurable set $A\subseteq X$ one has $\mu(S^{-1}(A))=\mu(A)$.  Such a measure $\mu$ is 
{\em ergodic} if $\mu$ can not be written in any non-trivial way as sum $\mu_1 + \mu_2$ of two invariant measures $\mu_1$ and $\mu_2$ (i.e. $\mu_1 \neq 0 \neq \mu_2$ and $\mu_1 \neq \lambda \mu_2$ for any $\lambda \in \R_{> 0}$). An invariant measure is called a {\em probability measure} if $\mu(X) = 1$, which is equivalent to $\underset{a_i \in \cal A}{\sum} \mu([a_i]) = 1$.
We denote by $\cal M(X)$ the set of invariant measures on $X$, and by 
$\cal M_1(X) \subset \cal M(X)$ the subset of probability measures.

The set $\cal M(X)$ 
is naturally equipped with an addition and an external multiplication with scalars $\lambda \in \R_{\geq 0}$. 
It is well known 
(see \cite{Walters}) that any invariant measure $\mu$ is determined by the values $\mu([w])$ 
for all $w \in \cal A^*$.
Hence 
the set $\cal M(X)$ is a convex linear cone which through $\mu \mapsto (\mu([w])_{w \in \cal A^*}$ is naturally embedded into the non-negative cone of the 
infinite dimensional 
vector space 
$\R^{\cal A^*}$. 
The cone $\cal M(X)$  is closed, and  the extremal vectors of $\cal M(X)$ are in 1-1 relation with the ergodic measures on $X$.
Furthermore, $\cal M_1(X)$ is compact, and it is the closed convex hull of its extremal points.
The following is well known (see
\cite{Walters}):

\begin{prop}
\label{Fer-Mont}
For any subshift $X \subset \Sigma_\cal A$
any family of ergodic measures $\mu_i \in \cal M(X)$, which are pairwise not scalar multiples of each other, is linearly independent.

In particular, if $X$ admits (up to scalar multiples) only finitely many ergodic measures, then $\cal M_1(X)$ is a finite simplex with vertices that are in 1 - 1 correspondence with the ergodic probability measures on $X$.
\qed
\end{prop}

It is well known that for any subshift $X \subset \Sigma_\cal A$ the set $\cal M(X)$ of invariant measures is not empty. If $\cal M_1(X)$ consists of a single point (which then must be ergodic), then $X$ is called {\em uniquely ergodic}.

%%%%%%%%%%%%%%

\subsection{Substitutions}
\label{substitutions}

${}^{}$

\begin{defn}
\label{substitution}
(1)
A {\em substitution} $\sigma$ is given by a map
$$\cal A \to \cal A^*, \,\, a_i \mapsto \sigma(a_i) \, .$$
A substitution defines both, an endomorphism of $\cal A^*$, and a continuous map from $\Sigma_\cal A$ to itself which maps $[w]$ to $[\sigma(w)]$.  Both of these maps are also denoted by $\sigma$, and both are summarized under the name of ``substitution''.

\smallskip
\noindent
(2)
If $\cal A$ and $\cal A'$ are two possibly distinct alphabets, then any monoid homomorphism $\sigma: \cal A^* \to \cal A'^*$ is also called a substitution, with the analogous convention for the induced map on $\Sigma_\cal A$.
\end{defn}

A substitution $\sigma: \cal A^* \to \cal A^*$ is called 
{\em everywhere growing} if each $a_i \in \cal A$ satisfies $|\sigma^n(a_i)| \to \infty$ for $n \to \infty$. 

For any substitution $\sigma$ we define the associated 
language $\mathcal L_\sigma \subseteq \cal A^*$ to be the set of factors 
of the words $\sigma^n(a_i)$, with $n\geq 1$ and $a_i\in\mathcal A$.

One defines 
the subshift $X_\sigma \subseteq \Sigma_\cal A$ {\em associated to the substitution $\sigma$} as the set of all $x = \ldots x_{k-1} x_k x_{k+1} \dots\in \Sigma_\cal A$ with the property 
that for any integers $m \geq n$ the word $x_{n} \ldots x_m$ is an element of $\mathcal L_\sigma$.

For any substitution $\sigma: \cal A^* \to \cal A'^*$ 
the non-negative matrix 
$$M_\sigma := (|\sigma(a)|_{a'})_{a' \in \cal A',\,  a \, \in \cal A}$$
is called the {\em incidence matrix} for the substitution $\sigma$ (to be specific: 
$a'$ gives the row 
index, while $a$ gives the column 
index of $M_\sigma$).
The substitution $\sigma$ is called {\em primitive} if $M_\sigma$ is primitive, i.e. there exists an integer 
$k \geq 1$ 
such that every coefficient of 
the power $M_\sigma^k$ is positive.

%%%%%%%%%%%%%%%%%%%%%%%%%%%
\subsection{$S$-adic sequences}
\label{S-adic-intro}

${}^{}$

In $S$-adic theory (see for instance \cite{BD, DLR}) one considers {\em directive sequences} of free monoids $\cal A^*_n$ and of monoid morphisms $\sigma_n: \cal A^*_{n+1} \to \cal A^*_{n}$ (for $n \geq 0$). The substitutions 
$\sigma_n$ 
belong to a given set $S$, which 
in many circumstances is 
assumed to be finite. 

We sometimes 
call $\cal A_n$ the {\em level alphabets}, and $\cal A_0$ the 
{\em base alphabet} of the directive sequence $\sigma$.
We also use the notation $\sigma_{[m, n)} = \sigma_m \circ \sigma_{m+1} \circ \sigma_{m+2} \circ \cdots \circ \sigma_n$ for any $n \geq m \geq 0$.
The directive sequence $\sigma$ is often represented by writing:
$$\sigma = \sigma_0 \circ \sigma_1 \circ \ldots$$

To any such a directive sequence $\sigma$
one associates the language $\cal L_\sigma \subset \cal A_0^*$, defined as the set of factors in $\cal A_0^*$ of the words $\sigma_0 \circ \sigma_1 \circ \ldots \circ \sigma_n(a_i)$, for any $n\geq 0$ and any $a_i\in\mathcal A_{n+1}$.
The subshift $X_\sigma \subseteq \Sigma_{\cal A_0}$ {\em associated to the directive sequence $\sigma$} is the set of all $x = \ldots x_{k-1} x_k x_{k+1} \dots\in \Sigma_{\cal A_0}$ 
such 
that for any two integers $m \geq n$ the word $x_{n} \ldots x_m$ is an element of $\mathcal L_\sigma$. The directive sequence $\sigma = \sigma_0 \circ \sigma_1 \circ \ldots$ is called an {\em $S$-adic expansion} of a subshift $X$ if $X = X_\sigma$ and if $S$ is a set of substitutions which contains every $\sigma_i$ that occurs in $\sigma$.

\medskip

The directive sequence $\sigma$ is called {\em everywhere growing} 
if one has
$$\min_{a_i \in \cal A_n}|\sigma_{[0, n)}(a_i)| \to \infty \qquad \text{for} \qquad n \to \infty \, .$$

One says that $\sigma$ is 
{\em weakly primitive} (or simply {\em primitive} by some authors) if for any $m \geq 1$ there is an integer $n \geq m+1$ such the incidence matrix $M_{\sigma_{[m, n)}}$ is positive. In this case it follows that $\sigma$ is everywhere growing (unless all level alphabets have cardinality 1).

This terminology coincides with that for substitutions introduced in subsection \ref{generalities}: indeed, one recovers the latter as special case of a 
{\em stationary} 
$S$-adic sequence, i.e. all terms $\sigma_n$ in the directive sequence $\sigma$ are equal.

\begin{prop}[\cite{BD}]
\label{weakly-primitive-BD}
For any weakly primitive 
directive sequence $\sigma$ the subshift $X_\sigma$ is minimal.
Furthermore, any minimal subshift $X$ admits an $S$-adic expansion that is weakly primitive.
\qed
\end{prop}

If the directive sequence $\sigma$ in the last proposition is stationary (or ``strongly minimal'', see Definition 5.1 of \cite{BD}), 
then one can deduce furthermore that $X_\sigma$ is uniquely ergodic. 

The hypothesis that our directive sequence $\sigma$ is everywhere growing is crucial to everything done in \cite{BHL1}; it will always be assumed. Fortunately this is not really a restriction, as is shown by the following elementary fact
(see Proposition 
5.10 
of \cite{BHL1}):

\begin{lem}
\label{any-lamination}
Let $\cal A$ be a finite alphabet, and let $X \subset \Sigma_\cal A$ be an arbitrary subshift. Then there exists an everywhere growing directive sequence $\sigma$ with 
base alphabet $\cal A_0 = \cal A$ such that
$X = X_\sigma$.
\qed
\end{lem}

The following seems to be well known (see \cite{BD}, Remark 5 and \cite{Dur}); a proof is provided through Corollary 
2.11 
of \cite{BHL1}:

\begin{fact}
\label{tower-bound}
(1)
For any directive sequence $\sigma$, 
where all level alphabets $\cal A_n$ are equal to some fixed 
alphabet $\cal A$ of cardinality $d \geq 1$, the number of distinct ergodic probability measures carried by the associated subshift $X_\sigma \subset \Sigma_\cal A$ is bounded above by $d$.

\smallskip
\noindent
(2)
In particular, the subset $\cal M_1(X_\sigma)$ of probability measures on $X_\sigma$ is a simplex of dimension $\dim \cal M_1(X_\sigma) \leq d-1$.
\qed
\end{fact}

%%%%%%%%%%%%%%%%%%%%

\section{Results from \cite{BHL1}}
\label{from-BHL1} 

\subsection{The general setting}

${}^{}$

Throughout this section we assume that $\sigma = \sigma_0 \circ \sigma_1 \circ \sigma_2 \circ \cdots$ is a directive sequence of substitutions $\sigma_n: \cal A^*_{n+1} \to \cal A^*_n$ as set up in the previous section. We also use the notation $\sigma_{[m, n)} = \sigma_m \circ \sigma_{m+1} \circ \sigma_{m+2} \circ \cdots \circ \sigma_n$ for any $n \geq m \geq 0$, and $M_n = M_{\sigma_n}$ and $M_{[m,n)} = M_{\sigma_{[m, n)}}$ for the associated incidence matrices.

Let $\cal A_1 = \{a_1, \ldots, a_d\}$ and $\cal A_2 = \{a'_1, \ldots, a'_{d'}\}$ be two alphabets, and let $\sigma: \cal A^*_2 \to \cal A_1^*$ be a substitution. We consider vectors $\vec v_1 = (\vec v_1(a))_{a \in \cal A_1}$ and $\vec v_2 = (\vec v_2(a'))_{a' \in \cal A_2}$ with real coordinates $\vec v_1(a) \geq 0$ and $\vec v_2(a') \geq 0$, and we say that $\vec v_1$ and $\vec v_2$ are {\em $\sigma$-compatible} if one has $\vec v_1 = M_{\sigma} \vec v_2$, where $M_{\sigma}$ denotes the incidence matrix of $\sigma$.

\begin{defn}
\label{vector-sequence}
Let $\sigma = \sigma_0 \circ \sigma_1 \circ \sigma_2 \circ \cdots$ be a directive sequence, and let 
$\bvec v = (\vec v_n)_{n \in \N \cup \{0\}}$ 
be a family of non-negative vectors $\vec v_n = (\vec v_n(a))_{a \in \cal A_n}$. We say that $\bvec v$ is a {\em $\sigma$-compatible vector tower} if for any $n \geq 0$ the vectors $\vec v_n$ and $\vec v_{n+1}$ are $\sigma_n$-compatible.
\end{defn}

We notice that there is a natural addition for $\sigma$-compatible vector towers, and similarly an external multiplication with non-negative scalars $\lambda \in \R_{\geq 0}$.
We are now able to state the main result of 
of our previous paper, 
translated properly into $S$-adic terminology:

\begin{thm}[{\cite{BHL1}}]
\label{thm1}
Let $\sigma = \sigma_0 \circ \sigma_1 \circ \sigma_2 \circ \cdots$ be an everywhere growing directive sequence with associated subshift $X_\sigma$. Let $\cal M := \cal M(X_\sigma)$ denote the set of invariant measures on $X_\sigma$, and let $\cal V = \cal V(\sigma)$ denote the set of $\sigma$-compatible vector towers.
\begin{enumerate}
\item
Every $\sigma$-compatible vector tower $\bvec v$ determines an invariant measure $\mu^{\tiny\bvec v}$ on $X_\sigma$.
\item
Conversely, every invariant measure $\mu$ on $X_\sigma$ is given via $\mu = \mu^{\tiny\bvec v}$ by some $\sigma$-compatible vector tower $\bvec v$.
\item
The issuing map 
$
{\frak m}
: \cal V \to \cal M, \bvec v \mapsto \mu^{\tiny\bvec v}$ is linear (with respect to linear combinations with non-negative scalars).
\item
For any word $w \in \cal A_0^*$
and any $\sigma$-compatible vector tower $\bvec v = (\vec v_n)_{n \in \N \cup \{0\}}$, 
with $\vec v_n = (\vec v_n(a))_{a \in \cal A_n}$, 
the sequence of sums 
$$\sum_{a \,\in \cal A_n} \vec v_n(a) \, |\sigma_{[0,n)}(a)|_w$$
is bounded above and increasing, 
and 
one has:
\begin{equation}
\label{approximation-sum}
\mu^{\tiny\bvec v}(
[w]) =\lim_{n \to \infty} \sum_{a \,\in \cal A_n} \vec v_n(a) 
\, 
|\sigma_{[0,n)}(a)|_w
\end{equation}
\end{enumerate}
\qed
\end{thm}

This is precisely 
the statement of Theorem 2.9 of \cite{BHL1}, except that the ``increasing'' property from (4) has been shown in Remark 9.5 of \cite{BHL1}. The canonical translation from the more general language of graph towers and vectors towers used in \cite{BHL1} into the traditional $S$-adic setting for subshifts in symbolic dynamics is explained in detail in section 3 of \cite{BHL1}.

\medskip

For any of the level alphabets $\cal A_n$ of a directive sequence $\sigma$ as above we consider the vector space $\R^{\cal A_n}$ and its non-negative cone $\R_{\geq 0}^{\cal A_n}$, as well as its image $\cal C_0^n := M_{[0,n)}(\R^{\cal A_n})$ in the base space $\R^{\cal A_0}$. From Definition \ref{vector-sequence} we obtain a canonical linear map 
$$
{\frak m}_0: \cal V(\sigma) \to \R^{\cal A_0}\, , \,\, 
\bvec v = (\vec v_n)_{n \in \N \cup \{0\}} \mapsto \vec v_0
\, ,$$
and it follows 
(see Proposition 10.2 (2) of \cite{BHL1}) 
that its image is equal to the nested intersection $\cal C_\infty := \bigcap \, \cal C_0^n$ of the cones $\cal C_0^n$. This gives (see \cite{BHL1}, Proposition 10.2 (1)):

\begin{lem}
\label{nested}
The map $\zeta: \cal M(X_\sigma) \to \R^{\cal A_0}, \, \mu \mapsto (\mu([a])_{a \in \cal A_0}$ satisfies ${\frak m}_0 = \zeta \circ {\frak m}$ and 
thus 
$\cal C_\infty = \zeta({\frak m}(\cal V(\sigma)))$. In particular, 
$\dim \cal C_\infty$ 
is a lower bound to the number of distinct ergodic probability measures on $X_\sigma$.
\qed
\end{lem}

Of special interest are directive sequences 
where every level alphabet $\cal A_n$ has the same cardinality $d \geq 1$, so that we can postulate them to be equal to $\cal A_n := \cal A = \{a_1 , \ldots, a_d\}$. 
In this case we say 
that $\sigma$ is a directive sequence {\em over $\cal A$}, and we say that $\sigma$ is of {\em tower dimension} $\dim \sigma = d$.

Examples are stationary sequences, or sequences derived through telescoping from directive sequences that have {\em finite tower dimension} $d$: the number $d$ is the inferior limit of the sequence of the $\,\card \cal A_n \in \N$.
(We believe that the notion of ``finite tower dimension'' is related or perhaps even equivalent to the condition ``finite rank'' as defined through the Bratteli-Vershik setting.)

%%%%%%%
\subsection{Application to substitutions}

${}^{}$

As pointed out in section \ref{prelims}, for any substitution $\sigma: \cal A^* \to \cal A^*$ the stationary directive sequence $\sigma_0\circ \sigma_1 \circ \sigma_2 \circ \ldots$, with $\sigma_n = \sigma$ for all $n \geq 0$, has as associated subshift the substitution subshift $X_\sigma$. This gives the possibility to interpret a compatible vector tower as infinite sequence of vectors in $\R_{\geq 0}^d$ (with 
$d = {\card \, \cal A}$), obtained from each other through iteration of the linear map $M_\sigma$. This observation has been used to derive in Theorem 10.8 of \cite{BHL1} the following result, which is a slight improvement of a result of 
Bezuglyi, Kwiatkowski, Medynets and Solomyak 
obtained in \cite{BKMS}:

\begin{prop}
\label{mieux-que-moulinette}
For any everywhere growing substitutions $\sigma$ the set of ergodic measures on the substitution subshift $X_\sigma$ is in 1-1 relation with the set of 
extremal vectors in the cone 
$$\cal C_\infty = \bigcap \{ M_\sigma^n (\R_{\geq 0}) \mid n \geq 1\}\, .$$ 
The latter are also the
non-negative extremal eigenvectors of 
a suitable positive power
$M^k_\sigma$ 
(for example $k = (\rm{card} \, A)!$ would do, see Appendix 11.3 of \cite{BHL1}).
\qed
\end{prop}

The determination of the extremal eigenvectors named in the above proposition is in practice for any given reducible matrix $M$ quite convenient, once one has penetrated the slightly intricate logic of the two ``conflicting'' 
natural 
partial orders on the primitive diagonal blocks of the power $M^k$. A concise description of all ingredients needed is given in Appendix 11.3 of \cite{BHL1}; for the convenience of the reader we will now single out the most frequently occurring non-primitive case:

\begin{cor}
\label{2-strata-case}
Let $\sigma$ be an everywhere growing substitution, and assume that the incidence matrix $M_\sigma$ satisfies the following conditions:
\begin{enumerate}
\item[(a)]
$M_\sigma$ is a $2 \times 2$ block lower triangular matrix.
\item[(b)]
The two diagonal blocks $M_{1,1}$ and $M_{2,2}$ are primitive, with Perron-Frobenius eigenvalues 
$\lambda_1 
\geq 
1$ and $\lambda_2 >1$ respectively.
\item[(c)]
The lower left off-diagonal block $M_{2,1}$ is non-zero.
\end{enumerate}

\smallskip
\noindent
(1)
If $\lambda_2 \geq \lambda_1$, then there is (up to scalar multiples) only one non-negative eigenvector $\vec v$ of $M_\sigma$ (with eigenvalue $\lambda_2$), which has zero-coordinates on the top block (corresponding to $M_{1,1}$), and non-zero coordinates on the bottom block (corresponding to $M_{2,2}$). In this case $X_\sigma$ has only one ergodic probability measure, and its support is the sub-subshift of $X_\sigma$ generated by the letters of $\cal A$ that define the bottom block.

\smallskip
\noindent
(2)
If $\lambda_1 > \lambda_2$, then there are (up to scalar multiples) precisely two non-negative eigenvectors $\vec v_1$ and $\vec v_2$ of $M_\sigma$. The vector $\vec v_2$ (with eigenvalue $\lambda_2$) has the same properties as the eigenvector $\vec v$ in case (1). On the other hand, the eigenvector $\vec v_1$ (with eigenvalue $\lambda_1$) is positive in all coordinates. 

In this case $X_\sigma$ has precisely two ergodic probability measures $\mu_1$ and $\mu_2$: The support of $\mu_2$ is, as in the above case (1), only the sub-subshift of $X_\sigma$ generated by the letters of $\cal A$ that define the bottom block. The support of $\mu_2$ is all of $X_\sigma$.
\qed
\end{cor}

%%%%%%%%%%%%%%%%%%%%
\subsection{Cylinder measures}

Let $\sigma = \sigma_0 \circ \sigma_1 \circ \sigma_2 \circ \cdots$ be a directive sequence of substitutions $\sigma_n: \cal A^*_{n+1} \to \cal A^*_n$, 
and let $\bvec v = (\vec v_n)_{n \in \N \cup \{0\}}$ be a $\sigma$-compatible vector tower as in Definition \ref{vector-sequence}. We now define, for any level $n \geq 0$ and any two letters $a, a' \in \cal A_n$, a {\em weight}
$\omega^n_{a, a'}$ through
equality (\ref{local-weights}) below. The latter can be viewed a special case of equality (\ref{approximation-sum}), thus ensuring the existence of the limit on the right hand side: 
\begin{equation}
\label{local-weights}
\omega^n_{a, a'} := \lim_{k \to \infty} 
\sum_{b \,\in \cal A_k} \vec v_k(b) \,\,  |\sigma_{[n,k)}(b)|_{a a'}
\end{equation}
Here the pair $(a, a')$ has to be understood as the ``transition'' from $a$ to $a'$, and by its ``$\sigma_{n-1}$-image'' $(c, c')$ we understand correspondingly the transition of the last letter $c$ of $\sigma_{n-1}(a)$ to the first letter $c'$ of $\sigma_{n-1}(a')$. We write $(c, c') = \sigma_{n-1}^+(a, a')$.

Of course, such transitions occur also inside $\sigma_{n-1}(a)$ or $\sigma_{n-1}(a')$, and indeed, we derive from (\ref{local-weights}) and from the $\sigma$-compatibility of $\bvec v$:
\begin{equation}
\label{weight-transition}
\omega^{n-1}_{c, c'} = \sum_{\{(a, a') \,\mid\, (c, c') = \sigma_{n-1}^+(a, a')\}} \omega^n_{a, a'} \,\, + \,\, \sum_{a \in \cal A_n}  \vec v_n(a) \,\,  |\sigma_{n}(a)|_{c c'}
\end{equation}

\medskip

The following has not been stated explicitly in \cite{BHL1}; 
we 
will 
though derive it quickly from 
the set-up studied there:

\begin{cor}
\label{cylinder-computation}
Let $\sigma$ and $\bvec v$ be as above.
Then, for any word $w \in \cal A_0^*$, 
and any integer $n\geq 0$ 
with 
$|\sigma_{[0,n)}(a)| \geq |w|-1$ for all $a \in \cal A_n$,
one obtains: 
$$
\begin{array}{cl}
\mu^{\tiny\bvec v}([w]) = & \underset{a \,\in \cal A_n}{\sum} \vec v_n(a) \cdot |\sigma_{[0,n)}(a)|_w \,\,+\\
                  & \underset{a, a' \,\in \cal A_n}{\sum} \omega^n_{a, a'} \, (|\sigma_{[0,n)}(aa')|_w - |\sigma_{[0,n)}(a)|_w - |\sigma_{[0,n)}(a')|_w)                
\end{array}
$$

\end{cor}

\begin{proof}
In \cite{BHL1} the vector towers $\bvec v$ from Theorem \ref{thm1} are defined by means of {\em weight function} $\omega_n$ on 1-vertex graphs $\Gamma_n$ that topologically realize the monoid $\cal A_n$. 
For any two letters $a, a' \in \cal A_n$ there is a {\em local edge} $\epsilon_{a, a'}$ defined in $\Gamma_n$, and if $\sigma$ is everywhere growing, then the vector tower $\bvec v$ in turn determines the weight functions $\omega_n$. In particular, see Remark 
9.5 
of \cite{BHL1}, the value of $\omega_n(\epsilon_{a, a'})$ is given via $\omega_n(\epsilon_{a, a'}) := \omega^n_{a, a'}$ through equality (\ref{local-weights}).
One obtains now the claimed statement as a direct translation of Propositions 
6.9 
and 7.4 from \cite{BHL1} into the $S$-adic language used here, following the instructions carefully laid out in section 3 of \cite{BHL1}.
\end{proof}

%%%%%%%%%%%%%%%%%%%%

\section{Minimal subshifts with many ergodic measures}
\label{d-measures1}

This section is devoted to the proof of Proposition \ref{d-measures}. We give first in subsection \ref{prop(1)} the proof of part (1) of this proposition, and improve this coarser approach in subsection \ref{prop(2)} to a proof of part (2).

\subsection{The general construction procedure}
\label{prop(1)}

${}^{}$

In this subsection we present (in a purposefully concrete and ``simplistic'' way) 
in 7 steps
a construction of a directive sequence $\sigma$ with the desired properties: 

\medskip
\noindent
(1)
We first consider matrices of type $M_\epsilon =  I_d + \epsilon 1_{d\times d}$, where $\epsilon > 0$, $I_d$ denotes the $(d \times d)$-unit matrix, and $1_{d\times d}$ denotes the $(d\times d)$-matrix that has all coefficients equal to 1. 

Let $\vec c \in \R^d$ 
denote the column vector with all coefficients equal to $1$, and let $\vec e = \vec e_k$ be any one of the standard basis vectors 
of $\R^d$. We compute:
$$
M_\epsilon \cdot \vec e = \vec e + \epsilon \vec c
\qquad \text{and} \qquad
M_\epsilon \cdot \vec c = \vec c + d \epsilon \vec c = (1 + d \epsilon) \vec c
$$

As a consequence, for any $\epsilon_1, \ldots, \epsilon_q > 0$ we obtain:

\begin{equation*}
\begin{split}
M_{\epsilon_1} \cdot M_{\epsilon_2} \cdot \ldots \cdot M_{\epsilon_q} \cdot \vec e
=
\vec e + [\epsilon_1 + (1+ d \epsilon_1) \epsilon_2 + (1+ d \epsilon_1) (1+ d \epsilon_2) \epsilon_3 \\+ \ldots +
(1+ d \epsilon_1) (1+ d \epsilon_2) \ldots (1+ d \epsilon_{q-1}) \epsilon_q] \vec c\\
= \vec e + (L_1 \epsilon_1 + L_2 \epsilon_2 + L_3 \epsilon_3 + \ldots + L_q \epsilon_q) \vec c\,,
\end{split}
\end{equation*}
where we 
set 
$L_1 := 1$ and $L_s := (1+ d \epsilon_1) (1+ d \epsilon_2) \ldots (1+ d \epsilon_{s-1})$ for any $s \in \{2, \ldots q\}$.

\medskip
\noindent
(2)
We now choose for every $n \in \N$ a value $\epsilon_n > 0$ which is small enough so that it satisfies
\begin{enumerate}
\item[(a)]
$\log (1+ d \epsilon_n) < 2^{-n} \log 2$ and
\item[(b)]
$\epsilon_n < \frac{1}{ 2^{n+1}}$\,.
\end{enumerate}
From (a) we deduce $L_n < 2$, 
so that we compute:
$$
M_{\epsilon_1} \cdot M_{\epsilon_2} \cdot \ldots \cdot M_{\epsilon_q} \cdot \vec e
=
\vec e + K_q \vec c
$$
with $0 < K_q \leq K_{q+1}  < 1$ for all $q \in \N$.

We note that the same result stays valid if one further lowers the value of the $\epsilon_n$, so that we can assume that $\epsilon_n = \frac{1}{\ell(n)}$ for some integer $\ell(n) \in \N$.

\medskip
\noindent
(3)
For the family of $\epsilon_n$ as in (2) we 
consider any of the standard basis vectors $\vec e_k$ and obtain 
$$ \lim_{n \to \infty} 
M_{\epsilon_1} \cdot M_{\epsilon_2} \cdot \ldots \cdot M_{\epsilon_n} \cdot \vec e_k
= \vec e_k + K_\infty \vec c
$$
for some $0 < K_\infty \leq 1$.

This shows that the nested intersection of the cones 
$M_{\epsilon_1} \cdot M_{\epsilon_2} \cdot \ldots \cdot M_{\epsilon_n}(\R^d_{\geq 0})$ for all $n \in \N$ is equal to the cone generated by the vectors $\vec e_k + K_\infty \vec c$, which is simplicial of dimension $d$.

\medskip
\noindent
(4)
We now define for any $\ell \in \N$ the integer matrix $M'_\ell = \ell I_d + 1_{d \times d}$, and note that for $\epsilon = \frac{1}{\ell}$ we have
$$M'_{\ell} = \ell M_{\epsilon}$$
It follows, since in (2) we chose $\ell(n) \in \N$ such that $\epsilon_n = \frac{1}{\ell(n)}$, that the nested intersection of the cones 
$M'_{\ell(1)} \cdot M'_{\ell(2)} \cdot \ldots \cdot M'_{\ell(n)}(\R^d_{\geq 0})$ for all $n \in \N$ is also equal to the cone generated by the vectors $\vec e_k + K_\infty \vec c$ and thus simplicial of dimension $d$.

\medskip
\noindent
(5)
Next we 
consider an alphabet $\cal A = \{a_1, \ldots, a_d\}$ and substitutions 
$$\sigma_n: \cal A^* \to \cal A^*, \quad a_k \mapsto a_k^{\ell(n)} a_1 \ldots a_d \quad (\text{for all  } a_k \in \cal A) \, ,$$
so that one has $M_{\sigma_n} = M'_{\ell(n)}$ for all $n \in \N$.

For formal reasons we add the substitution $\sigma_0 := {\rm id}_{\cal A^*}$ to our list.

\medskip
\noindent
(6)
Since 
for $n \geq 1$ any 
of the incidence matrices $M_{\sigma_n}$ is positive, 
it follows that the directive sequence $\sigma = \sigma_0 \circ \sigma_1 \circ \sigma_2 \circ \ldots$ is weakly primitive 
as defined in section \ref{prelims},
so that Proposition \ref{weakly-primitive-BD} shows that the associated subshift $X_\sigma$ is minimal.

\medskip
\noindent
(7)
We now apply Theorem \ref{thm1}, Lemma \ref{nested} and Fact \ref{tower-bound} to deduce 
that
$\cal M(X_\sigma)$ is a simplicial cone of dimension $d$, so that
$X_\sigma$ 
supports $d$ distinct ergodic probability measures.

\begin{rem}
\label{variations}
Once the matrices $M'_{\ell(n)}$ have been established in step (4) of the above construction, the choice of the substitutions $\sigma_n$ with $M_{\sigma_n} = M'_{\ell(n)}$ in step (5) is of course only one of many possible such choices. 
\end{rem}

%%%%%%%%%%%%%%%%%%%
\subsection{Reduction to finite $S$}
\label{prop(2)}

${}^{}$

We will now explain how the method presented in the previous subsection can be refined to obtain the same result, but for a 
set $S$ of substitutions 
that is finite.
The latter are given by the substitutions $\rho_j$,  $\theta_j$ and $\tau_j$, for any $j \in \{1, \ldots, d\}$, which all fix any $a_k$ with $k \neq  j$ and map $a_j$ to $\rho_j(a_j) = a_j^3$, $\theta_j(a_j) = a_j^2$ and 
$\tau_j(a_j) = a_j a_{j+1} \ldots a_d a_1 \ldots a_{j-1}$ 
respectively. We also use the abbreviation $\theta'_j := \theta_d \ldots \theta_{j+1} \theta_{j-1} \ldots \theta_1$.

\smallskip

We first need to recall the following well known fact (deduced easily from the non-commensurability of the logarithms of $2$ and $3$):

\begin{fact}
\label{needed}
(1)
The set of numbers $\frac{2^m}{3^q}$, for any integers $m, q \geq 0$, is dense in $\R_{\geq 0}$.

\smallskip
\noindent
(2)
More concretely, we need the following, which follows from (1):

For any $\epsilon > 0$ there are integers $m, q \geq 0$ such that $1 \geq \frac{2^m}{3^q} \geq 1 - \frac{\epsilon}{2}$. We then define the integer $h \geq 1$ through $h := 3^q - 2^m$.
\qed
\end{fact}

For any $\epsilon > 0$ and $m, q$ and $h$ as in Fact \ref{needed} (2) we now define a substitution 
$$\sigma'_{\epsilon,j} := \rho^{q}_j \tau^{h}_j {\theta'_j}^{m}$$
and observe that on the generators $\sigma'_{\epsilon, j}$ 
acts as follows:
$$ 
a_j \mapsto a_j^{3^{q}} (a_{j+1} \ldots a_d a_1 \ldots a_{j-1})^{{h}}$$
$$a_k \mapsto a_k^{2^{m}} \qquad \text{for \,\, all} \qquad k \neq j$$

From the above definition of $h$ we quickly deduce that the incidence matrix $M_{\sigma'_{\epsilon,j}} =: M'_{\epsilon, j}$ maps the center vector $\vec c$ from the previous section to
$$M'_{\epsilon, j} \cdot \vec c = 3^q \vec c\, ,$$
and similarly any basis vector $\vec e_k$ with $k \neq j$ to
$$M'_{\epsilon, j} \cdot \vec e_k = 2^m \vec e_k\, .$$
For $\vec e_j$ we compute: 
$$M'_{\epsilon, j} \cdot \vec e_j = h \vec c + (3^q - h) \vec e_j = h \vec c + 2^m \vec e_j$$

\begin{lem}
\label{include}
For any $1 \geq \epsilon > 0$ the matrices $M'_{\epsilon, j}$ above and $M_\epsilon$ from the previous subsection satisfy the following inclusion:
$$M_\epsilon(\R^d_{\geq 0}) \subset M'_{\epsilon, j} (\R^d_{\geq 0})$$
\end{lem}

\begin{proof}
We recall from the previous subsection that
 $M_\epsilon \cdot \vec c =  (1 + d \epsilon) \vec c$ and 
$M_\epsilon \cdot \vec e_k = \vec e_k + \epsilon \vec c$ for all $k \in \{1, \ldots, d\}$.
We note that both, $M_\epsilon$ and $M'_{\epsilon, j}$ fix the vector $\vec c$ up to a scalar multiple and map the vector $\vec e_k$, for any $k \in \{1, \ldots, d\}$, to a linear combination of $\vec c$ and $\vec e_k$. Hence the claim follows if we show for any $k$
that the projectivized line segment $[\vec c, \vec e_k]$ is mapped by $M'_{\epsilon, j}$ to a larger segment than by $M_\epsilon$. 

For $k \neq j$ this is clear, since the projectivized line segment $[\vec c, \vec e_k]$ is set-wise fixed by $M'_{\epsilon, j}$.
For $k = j$ we compute the coefficient quotient of $M'_{\epsilon, j} \cdot \vec e_j$ 
and apply the inequalities from Fact \ref{needed} (2) to obtain:
$$\frac{h}{2^m} = \frac{3^q - 2^m}{2^m} = \frac{3^q}{2^m} - 1 \leq \frac{2}{2-\epsilon} - 1 = \frac{\epsilon}{2-\epsilon}$$
In comparison, the analogous coefficient quotient of $M_\epsilon \cdot \vec e_j$  is equal to $\epsilon$ and thus strictly bigger, since we can assume that $\epsilon$ is small. This shows that $M_\epsilon \cdot \vec e_j$ defines a point on the $M'_{\epsilon, j}$-image of the projectivized line segment $[\vec c, \vec e_j]$.
\end{proof}

We now consider an infinite sequence of positive constants $\epsilon_1, \epsilon_2, \ldots$ as exhibited in the previous subsection, for which the limit intersection of the nested sequence of the cones $M_{\epsilon_1} \cdot M_{\epsilon_2} \cdot \ldots \cdot M_{\epsilon_n}(\R^d_{\geq 0})$ is simplicial of dimension $d$.

For any $\epsilon_n > 0$ let $m_n, q_n$ and $h_n$ be integers as in Fact \ref{needed} (2), and for any choice of indices $j(n) \in \{1, \ldots, d\}$ let 
$$\sigma'_{n,j(n)} := \rho^{q_n}_{j(n)} \tau^{h_n}_{j(n)} \theta'_{j(n)}\,^{m_n}$$
be the corresponding substitution as introduced above. We denote its incidence matrix by
$M'_n := M_{\sigma'_{n,j(n)}}$, 
and obtain directly from Lemma \ref{include} that for any $n \in \N$ one has:
$$M_{\epsilon_1} \cdot M_{\epsilon_2} \cdot \ldots \cdot M_{\epsilon_n}(\R^d_{\geq 0}) \subset M'_{1} \cdot M'_{2} \cdot \ldots \cdot M'_{n}(\R^d_{\geq 0})$$
It follows that the same inclusion is true for the limit cones, so that the limit cone for the $\sigma'_{n,j(n)}$ must also have dimension $d$.

It remains to specify the indices $j(n)$: if they are chosen so that $j(n)$ varies cyclically through the set of all indices $1, 2, \ldots, d$, then the product of any subsequent $d$ incidence matrices is positive, which suffices to guaranty that the resulting directive sequence is weekly primitive. 

\smallskip

It follows as in the previous subsection that the $S$-adic subshift defined by the directive sequence of $\sigma'_{n,j(n)}$, for the finite set $S$ given above, is minimal and has $d$ ergodic probability measures.

\begin{rem}
\label{more-general}
Similar to Remark \ref{variations} we note that there are many possible variations of the construction presented in this subsection. For examples the exponents of the powers $\rho_j(a_j)$ and $\theta_j(a_j)$ can be arbitrary relative prime integers, and $\tau_j(a_j)$ can be any word that involves every letter of $\cal A$ a fixed number of times.
\end{rem}

\begin{rem}
\label{5-subs}
We now observe that any two of the substitutions $\rho_j$ and $\rho_{j+1}$ are conjugate to each other, through a cyclic permutation $\pi: a_j \mapsto a_{j+1}$ (for $j$ understood modulo $d$). 
The same holds for the $\theta_j$ and for the $\tau_j$. 
It follows 
that the above constructed directive sequence $\sigma' := \sigma'_{0,j(0)} \circ \sigma'_{1,j(1)} \circ \ldots$ can be understood as obtained through telescoping from a directive sequence $\sigma'' = \sigma''_{0} \circ \sigma''_{1} \circ \ldots$ with same associated subshift $X_{\sigma'} = X_{\sigma''}$, where all $\sigma''_n$ belong to $S = \{\rho_1, \theta_1, \tau_1, \pi\}$. This proves the statement (2) of Proposition \ref{d-measures}.
\end{rem}

\begin{rem}
\label{comment-referee}
The relevance of statement (2) of Proposition \ref{d-measures} is emphasized by a comparison with F. Durand's results in \cite{Dur}, where it is shown that a subshift with Bratteli-Vershik representation based on a finite set of positive incidence matrices is linearly recurrent and hence uniquely ergodic.

In our construction above it is crucial that the 4 generating substitutions $\rho_1, \theta_1, \tau_1, \pi$ are not positive (indeed, not even primitive), so that suitable products of them, as for example the ones pointed out in the above proof, allows one to compose an infinite set of positive matrices without (linear or else) bound on the resulting recurrencence.
\end{rem}

%%%%%%%%%%%%%%%%
\subsection{Some questions}
\label{questions-chapter4}

${}^{}$

Inspired by Remark \ref{5-subs} we define 
for any set $S$ 
of substitutions of some free monoid $\cal A^*$
the {\em ergodic size} $\eta(S)$ 
as 
the maximal number of ergodic probability measures supported by any subshift $X \subset \Sigma_\cal A$ which admits an $S$-adic expansion. One can specify this further to a {\em min-ergodic size} $\eta_{\rm min}(S)$ by adding the additional requirement that $X$ be minimal. This gives 
$$
1 \leq \eta(S) \leq \card \cal A
\quad \text{and} \quad\eta_{\rm min}(S) \leq \eta(S)
$$
for any $S \neq \emptyset$, as well as
$$\eta_{\rm min}(S') \leq \eta_{\rm min}(S) \quad \text{and} \quad\eta(S') \leq \eta(S)$$
for any subset $S' \subset S$.
We say that $S$ is 
{\em ergodically rich} if $\eta_{\rm min}(S) \geq 2$; if $\eta(S) = 1$, we call $S$ {\em uniquely ergodic}.

\smallskip

In Corollary 10.10 of \cite{BHL1} it has been shown that for any directive sequence $\sigma$ with stationary incidence matrix $M$ the number of ergodic probability measures on the associated subshift $X_\sigma$ depends only on $M$ and not on the particular choice of $\sigma$. It follows that for the finite set $S_M$ of substitutions $\sigma_i$ with $M_{\sigma_i} = M$ the ergodic size satisfies $\eta(S) = 1$, if $M$ is chosen suitably (for example primitive). 

On the other hand, 
the min-ergodic size of the 4-elements set $S$ exhibited in Remark \ref{5-subs} is shown there to satisfy $\eta_{\rm min}(S) = d$, thus motivating the following:

\begin{question}
\label{ergodic-size}
(1)
Given integers $k \geq 1$ and $d \geq 1$, what are the ``generic'' values for $\eta(S)$ and $\eta_{\rm min}(S)$
of any set $S$ that consists of at most $k$ substitutions over an alphabet of $d$ letters ?

\smallskip
\noindent
(2)
More specifically, does there exist an ergodically rich set $S$ that consists of 2 primitive substitutions ?
\end{question}

%%%%%%%%%%%%%%%

\section{Determination of cylinder weights for substitutions}
\label{cylinder-determination}

Let 
$\cal A = \{a_1, \ldots, a_d\}$ 
be an alphabet, and let $\sigma: \cal A^* \to \cal A^*$ be a substitution which is everywhere growing, but possibly reducible. As before we denote by $X_\sigma \subset \Sigma_\cal A$ the subshift associated to $\sigma$, and by $\mu$ an invariant measure on $X_\sigma$. Recall that for any $w \in \cal A^*$ the cylinder determined by $w$ is denoted by $[w]$.

\begin{question}
\label{cylinder-measure-quest}
How can one determine the measure $\mu([w])$ for an arbitrary word $w \in \cal A^*$ ?
\end{question}

We will now describe a (fairly practical) answer to this question, based on our previous results. 
Along the description of this algorithm, we will illustrate the main steps on the two (famous) examples of the Thue-Morse substitution 
$\sigma_{\rm TM}$ and the Fibonacci substitution $\sigma_{\rm Fib}$:
$$\begin{array}{rcllrcl}
\sigma_{\rm TM}: a & \mapsto & ab & \qquad & \sigma_{\rm Fib}: a & \mapsto & ab \\
                 b & \mapsto & ba &             &                  b & \mapsto & a 
\end{array}
$$

The algorithm we 
present 
consists of three steps:

\smallskip
\noindent
\underline{\bf Step 1:} We first compute from the given words $\sigma(a_i)$ the incidence matrix $M_{\sigma} = (m_{x,y})_{x, y \in \cal A}$, where $m_{x, y} := |\sigma(y)|_x$ denotes the number of occurrences of $x$ in the word $\sigma(y)$. 

We then pass to the {\em augmented incidence matrix} 
$$M^+_{\sigma} : = (m^+_{X,Y})_{X, Y \in 
\cal A_2}$$
which is defined 
as follows: Its rows and columns are indexed by the set 
$\cal A_2$
of all words in $\cal A^*$ of length 1 or 2, and it contains $M_{\sigma}$ as diagonal block: $m^+_{x, y} := m_{x,y}$ for all $x, y \in \cal A$. The complementary diagonal block is a 
``pre-permutation matrix'' (i.e. every column contains 
precisely 
one non-zero entry, and the latter is equal to 1), defined by the rule that for any word $X = x_1 x_2$ of length 2 the coefficient $m^+_{X,Y}$ is equal to $1$ if and only if for $Y = y_1 y_2$ the letter 
$x_1$ is the last letter of $\sigma(y_1)$ and the letter $x_2$ is the first letter of $\sigma(y_2)$. 
Otherwise one sets $m^+_{X,Y}=0$.

The off-diagonal blocks are defined by the rule that $m^+_{X,Y} = 0$ if $|X| = 1$ and $|Y| = 2$, while for $|X| = 2$ and $|Y| = 1$ one sets $m^+_{X,Y} := |\sigma(Y)|_X$.

\begin{example}
\label{augmented-matrix}
For the Thue-Morse substitution $\sigma_{\rm TM}$
and the Fibonacci substitution $\sigma_{\rm Fib}$
one has $M_{\sigma_{\rm TM}} = 
\left[
\begin{array}{cc}
1 & 1 \\
1 & 1
\end{array} \right]$
and
$M_{\sigma_{\rm Fib}} = 
\left[
\begin{array}{cc}
1 & 1 \\
1 & 0
\end{array} \right]$, 
and 
we obtain, for the row and column indexes given by $(a, b, aa, ab, ba, bb)$:
$$
M^+_{\sigma_{\rm TM}} = 
\left[
\begin{array}{cc|cc:cc}
1 & 1 & 0 & 0 & 0 & 0 \\
1 & 1 & 0 & 0 & 0 & 0 \\
\hline
0 & 0 & 0 & 0 & 1 & 0 \\
1 & 0 & 0 & 0 & 0 & 1 \\
\hdashline
0 & 1 & 1 & 0 & 0 & 0 \\
0 & 0 & 0 & 1 & 0 & 0 
\end{array} 
\right],
\qquad
M^+_{\sigma_{\rm Fib}} = 
\left[
\begin{array}{cc|cc:cc}
1 & 1 & 0 & 0 & 0 & 0 \\
1 & 0 & 0 & 0 & 0 & 0 \\
\hline
0 & 0 & 0 & 0 & 1 & 1 \\
1& 0 & 0 & 0 & 0 & 0 \\
\hdashline
0& 0 & 1 & 1 & 0 & 0 \\
0 & 0 & 0 & 0 & 0 & 0 
\end{array} 
\right].$$
\end{example}

\begin{rem}
\label{powers}
It is not hard to verify that with this definition one obtains for any two substitutions $\sigma, \sigma'$ the equality $M^+_{\sigma} \cdot M^+_{\sigma'} = M^+_{\sigma \circ \sigma'}$. Thus we have in particular, for any integer $n \in \N$:
$$M^+_{\sigma^n} = 
({M^+_{\sigma}})^n$$
\end{rem}

\begin{rem}
\label{Kronecker}
The augmented incidence matrix $M_\sigma^+$ is 
lower triangular by blocks, with two diagonal blocks. 
The first diagonal block is given by the incidence matrix $M_\sigma$. 
The second diagonal block can be described as a Kronecker product 
$S_\sigma\otimes P_\sigma$.
For this, we order lexicographically the elements of 
$\cal A_2 \smallsetminus \cal A$
which serve as indices for the second diagonal block.
Now 
$P_\sigma=(p_{xy})_{x,y\,\in\mathcal A}$ is 
the ``prefix matrix'' of $\sigma$: we set
$p_{xy}=1$ if $x$ is the first letter of $\sigma(y)$, and otherwise $p_{xy}=0$.
Similarly, $S_\sigma=(s_{xy})_{x,y\,\in\mathcal A}$ is ``suffix matrix'' of $\sigma$: 
$s_{xy}=1$ if $x$ is the last letter of $\sigma(y)$, and otherwise $s_{xy}=0$.

As a consequence, we 
notice that
the spectrum of $M_\sigma^+$ is given by the spectrum of $M_\sigma$ plus possibly $0$ and some $d'$-th roots of the unity, for $d' \leq d$.
\end{rem}

\begin{example}
For the Thue-Morse substitution $\sigma_{\rm TM}$ and the Fibonacci substitution $\sigma_{\rm Fib}$, we get:
$$
S_{\sigma_{\rm TM}} = 
\left[
\begin{array}{cc}
0 & 1 \\
1 & 0
\end{array} 
\right],
\;\;
P_{\sigma_{\rm TM}} = 
\left[
\begin{array}{cc}
1 & 0 \\
0 & 1
\end{array} 
\right],
\;\;
S_{\sigma_{\rm Fib}} = 
\left[
\begin{array}{cc}
0 & 1 \\
1 & 0
\end{array} 
\right],
\;\;
P_{\sigma_{\rm Fib}} = 
\left[
\begin{array}{cc}
1 & 1 \\
0 & 0
\end{array} 
\right].
$$
\end{example}

\medskip
\noindent
\underline{\bf Step 2:}
For any word $w \in \cal A^*$ 
we define 
an integer
$n \geq 0$ 
to be {\em $(\sigma, w)$-large} if 
$|\sigma^n(a_i)|\geq |w|- 1$ 
for all $a_i \in \cal A$. For such $n$ we now define 
an {\em occurence vector} 
$\vec v(w)_n = (v_X)_{X \in 
A_2}$ 
with coefficients $v_X$ defined by the rule 
$v_X = |\sigma^n(X)|_w$ if $|X| = 1$, and 
$v_X = |\sigma^n(X)|_w - |\sigma^n(x_1)|_w - |\sigma^n(x_2)|_w$ 
if $X = x_1 x_2$ has length 2. 
Thus, in the latter case, the coefficient $v_X$ counts the number of times that $w$ occurs as factor of 
$\sigma^n(x_1) \sigma^n(x_2)$, 
but ignores those occurrences that are either completely contained in the first or completely contained in the second factor of this product.

\begin{rem}
\label{word-vector}
(1)
If $w$ has length 1 or 2, then independently of $\sigma$ one has that $n = 0$ is $(\sigma, w)$-large. The vector $\vec v(w)_0$ has only zero coefficients except for the coordinate $X = w$, where it is equal to 1.

\smallskip
\noindent
(2)
It is not hard to verify (using Remark \ref{powers}) that the above defined {$n$-th occurrence vector} $\vec v(w)_n$ satisfies, for any 
two $(\sigma, w)$-large integers $n \geq m$
the equality
\begin{equation}
\label{n-minus-m}
^t\vec v(w)_n = \,^t\vec v(w)_m  \cdot (M^+_{\sigma})^{n-m} \, .
\end{equation}
\end{rem}

\begin{example}
\label{occurrence-vectors}
(1)
We now consider the substitutions $\sigma_{\rm TM}$ and $\sigma_{\rm Fib}$ from Example \ref{augmented-matrix}
and any word $w$ of length $|w| = 2$. We know from Remark \ref{word-vector} (1) that any $n \geq 0$ is $(\sigma_i, w)$-large, and as warm-up exercise we consider $n = 1$ for $\sigma_{\rm TM}$ and $n = 2$ for $\sigma_{\rm Fib}$.
For the same ordering of coordinates as in Example \ref{augmented-matrix} this gives for $\sigma_{\rm TM}$ the following occurrence vectors: 
\[ \begin{array}{|c|c|c|c|}
\hline
^t\vec v(aa)_1&^t\vec v(ab)_1&^t\vec v(ba)_1&^t\vec v(bb)_1\\
\hline
 (0, 0, 0, 0, 1, 0) &(1, 0, 0, 0, 0, 1)&(0, 1, 1, 0, 0, 0)& (0, 0, 0, 1, 0, 0)\\
 \hline
\end{array}\]

For $\sigma_{\rm Fib}$ we obtain the vectors 
\[ \begin{array}{|c|c|c|c|}
\hline
^t\vec v(aa)_2&^t\vec v(ab)_2&^t\vec v(ba)_2&^t\vec v(bb)_2\\\hline
 (0, 0, 1, 1, 0, 0) &(1, 1, 0, 0, 0, 0)&(1, 0, 0, 0, 1, 1)& (0, 0, 0, 0, 0, 0)\\
 \hline
\end{array}\]
For both sets of occurrence vectors one uses Remark \ref{word-vector} (1) to quickly verify the equation (\ref{n-minus-m}), for $m = 0$ and $n$ as above.

\smallskip
\noindent
(2) We now want to illustrate how to deal with a more elaborate case, by picking ``at random'' a slightly longer word:
$$w = baabab$$
One first iterates the given substitution on the generators, until their images are longer or equal to $|w|-1 = 5$. In our two cases this gives:
$$\begin{array}{rcllrcl}
\sigma^3_{\rm TM}: a & \mapsto & abbabaab & \qquad & \sigma^4_{\rm Fib}: a & \mapsto & abaababa \\
                        b & \mapsto & baababba &             &                  b & \mapsto & abaab 
\end{array}
$$
We now compute for $\sigma_{\rm TM}$: 
$$
v_a := |\sigma_{\rm TM}^3(a)|_w = 0, \,\,\, v_b := |\sigma_{\rm TM}^3(b)|_w = 1, \,\,\,
v_{aa} := |\sigma_{\rm TM}^3(aa)|_w = 1, \,\,\, 
$$
$$
v_{ab} := |\sigma_{\rm TM}^3(ab)|_w = 1, \,\,\, v_{ba} := |\sigma_{\rm TM}^3(ba)|_w = 1, \,\,\,  v_{bb} := |\sigma_{\rm TM}^3(bb)|_w = 2
$$
Similarly, we obtain for $\sigma_{\rm Fib}$: 
$$
v_a := |\sigma_{\rm Fib}^4(a)|_w = 1, \,\,\, v_b := |\sigma_{\rm Fib}^4(b)|_w = 0, \,\,\,
v_{aa} := |\sigma_{\rm Fib}^4(aa)|_w = 2, \,\,\, 
$$
$$
v_{ab} := |\sigma_{\rm Fib}^4(ab)|_w = 1, \,\,\, v_{ba} := |\sigma_{\rm Fib}^4(ba)|_w = 2, \,\,\,  v_{bb} := |\sigma_{\rm Fib}^4(bb)|_w = 1
$$
We then use the general formula for the occurrence vector
$$
^t\vec v(w)_n = (v_a, v_b, v_{aa} - 2 v_a, v_{ab} - v_a - v_b, v_{ba} - v_b - v_a, v_{bb} - 2 v_b) 
$$
to obtain for $\sigma_{\rm TM}$
$$
^t\vec v(aabab)_3 = (0, 1, 1, 0, 0, 0)\, ,
$$
and for $\sigma_{\rm Fib}$
$$
^t\vec v(aabab)_4 
= (1, 0, 0, 0, 1, 1) \, .
$$
\end{example}

\smallskip
\noindent
\underline{\bf Step 3:}
We know from Proposition \ref{mieux-que-moulinette} 
that every ergodic invariant measure $\mu_i$ on the substitution subshift $X_\sigma \subset \Sigma_\cal A$ is given 
for some suitable exponent $k \geq 1$
by a non-negative column 
eigenvector $\vec v_i = (v_x)_{x \in \cal A}$ of $M^k_{\sigma}$ with eigenvalue $\lambda_i > 1$, 
with the property $v_x = \mu([x])$ for any $x \in \cal A$ (see Lemma \ref{nested}). The converse statement also holds.

Since for the matrix $M^+_\sigma$ all eigenvalues $\lambda_j$ 
of the lower $d^2 \times d^2$ diagonal block are of modulus $|\lambda_j| \leq 1$, it follows from standard Perron-Frobenius theory for non-negative matrices (see Appendix 11.3 of \cite{BHL1}) 
that any eigenvector $\vec v_i$ of $M_\sigma^k$ as above 
determines uniquely a non-negative eigenvector $\vec v^+_i$ of $(M^+_{\sigma})^k$, which contains $\vec v_i$ as ``subvector'' at the upper $d$ coordinates (i. e. the coordinates for $\cal A \subset \cal A_2$), and which has the same eigenvalue $\lambda_i$.

We have thus explained all terms used in Proposition \ref{cylinder-formula}, which we abbreviate here to:

\begin{prop}
\label{cylinder-formula2}
For any convex combination $\mu = \sum c_i \mu_i$ of the above ergodic measures $\mu_i$ the measure of the cylinder $[w]$ is given by the scalar product
(written as matrix multiplication)
$$\mu([w]) = 
\,^t\vec v(w)_n \cdot \sum c_i \frac{1}{\lambda_i^n} \vec v^+_i \, .$$
\end{prop}

\begin{proof}
It suffices to prove $\mu_i([w]) = \,^t\vec v(w)_n \cdot\frac{1}{\lambda_i^n} \vec v^+_i $ for a single ergodic measure $\mu_i$. 
But this is precisely the formula from Corollary \ref{cylinder-computation} for the measure 
$\mu_i = \mu^{\tiny \bvec v_{\!\!  i}}$ 
given by 
a 
vector tower 
$\bvec v_{\! \! i}$  
that 
is compatible with the stationary directive sequence defined by the substitution $\sigma$, where $\bvec v_{\!\! i}$ is defined by the above eigenvector $\vec v_i$ through setting 
$\bvec v_{\! \!i} = (\frac{1}{\lambda_i^n} \vec v_i)_{n \in \N \cup\{0\}}$.
\end{proof}

\begin{rem}\label{rem-cylindres2}
Note 
that 
for any measure $\mu_i$ as above 
the associated right eigenvector $\vec v_i^+$ of the augmented matrix 
$M^+_\sigma$ 
gives directly the measures of the cylinders of size 
1 and 2. 
For the words of length 
1 this is 
is stated in Lemma \ref{nested}, and 
for the words $X$ of length 2 this follows from 
Proposition \ref{cylinder-formula2}, since $\vec v(X)_0$ is the unit vector for the coordinate $X$, and $0$ is $(\sigma, X)$-large for any everywhere growing $\sigma$ and $X$ of length 2. 

\end{rem}

\begin{example}
\label{formula-applied}
The two substitutions $\sigma_{\rm TM}$ and $\sigma_{\rm Fib}$ from Example \ref{augmented-matrix} are primitive and hence uniquely ergodic. 
For 
the Thue-Morse substitution 
$\sigma_{\rm TM}$ 
on the full shift $\Sigma_{\{a,b\}}$, 
with incidence matrix 
$M_{\sigma_{\rm TM}} =
\left[
\begin{array}{cc}
1 & 1 \\
1 & 1
\end{array} 
\right]$, 
the probability measure $\mu_{\rm TM}$ 
is defined by the Perron-Frobenius eigenvector 
$ \frac{1}{2}
\left[
\begin{array}{c}
1 \\
1
\end{array} 
\right]$
with eigenvalue $\lambda_{\rm TM} = 2$.
For the Fibonacci substitution $\sigma_{\rm Fib}$ 
with incidence matrix
$M_{\sigma_{\rm Fib}} =
\left[
\begin{array}{cc}
1 & 1 \\
1 & 0
\end{array} 
\right]$
we obtain analogously the measure $\mu_{\rm Fib}$ through 
the eigenvector 
$\frac{1}{\varphi^2}
\left[
\begin{array}{c}
\varphi \\
1
\end{array} 
\right]$
with eigenvalue $\lambda_{\rm Fib} = \phi$, for 
the golden mean 
$\phi := \frac{1 + \sqrt{5}}{2}$.
For the augmented incidence matrices $M^+_{\sigma_{\rm TM}}$ and $M^+_{\sigma_{\rm Fib}}$ computed in Example \ref{augmented-matrix} we thus calculate the augmented eigenvectors (written transposed) as 
$^t\vec v_{\rm TM}^+ = \frac{1}{2} \, ^t(1, 1, \frac{1}{3}, \frac{2}{3}, \frac{2}{3}, \frac{1}{3})$ 
and 
$^t\vec v_{\rm Fib}^+ = \frac{1}{\varphi^2} ^t(\varphi, 1, \varphi^{-1}, 1, 1, 0)$. 
Hence using 
Remark \ref{word-vector} (1) 
we can now evaluate the formula in Proposition \ref{cylinder-formula2} to obtain:

\[ \begin{array}{|c|c|c|c|c|c|c|}
\hline
 w&a&b&aa&ab&ba&bb\\
\hline
\rule [-0.3cm]{0mm}{0.8cm}
\mu_{\rm TM}([w])&\frac{1}{2}&\frac{1}{2}&\frac{1}{6}&\frac{1}{3}&\frac{1}{3}&\frac{1}{6}\\
 \hline
\end{array}\]

\noindent
as well as 

\[ \begin{array}{|c|c|c|c|c|c|c|}
\hline
 w&a&b&aa&ab&ba&bb\\
\hline
\rule [-0.3cm]{0mm}{0.8cm}
\mu_{\rm Fib}([w])&\frac{1}{\varphi}&\frac{1}{\varphi^2}&\frac{1}{\varphi^3}&\frac{1}{\varphi^2}&\frac{1}{\varphi^2}&0\\
 \hline
\end{array}\]

\medskip

It is not hard to check 
that Remark \ref{rem-cylindres2} gives the same answer, and is 
slightly more direct.
But 
in order 
to compute the measures of cylinders of size bigger or equal to 3, 
Remark \ref{rem-cylindres2} does not apply, 
so that 
one needs to fall back onto 
Proposition \ref{cylinder-formula2}. 
For example, for the word $w = baabab$ considered in Example \ref{occurrence-vectors} (3) we obtain 
$$
\mu_{\rm TM}([baabab]) = \frac{1}{8} \,^t\vec v(baabab)_3 \cdot \vec v_{\rm TM}^+ = 
\frac{1}{12}$$
and 
$$
\mu_{\rm Fib}([baabab]) = \frac{1}{\phi^4} \,^t\vec v(baabab)_4 \cdot \vec v_{\rm Fib}^+ = \frac{1}{\phi^4}\, .$$

For completeness we mention the fact 
that all the cylinders of length $3$ have the same measure for Thue-Morse subshift, or else they have measure 0.
\end{example}

%%%%%%%%%%%%%%%%%%%%

\section{Examples}
\label{section-examples}

We conclude this 
paper 
with three 
slightly more challenging 
examples, concerning non-primitive everywhere growing substitutions.

%%%%%%
\subsection{An example with a periodic sequence in the substitution subshift}
\label{first-example}

${}^{}$

The methods developed in \cite{BKMS} don't seem to work for substitutions
with periodic leaves in their associated subshift.
Such a restriction doesn't hold for the technology presented here, as is illustrated by the following.

Let $\sigma$ be the (non-primitive) substitution given by:
$$\begin{cases} a\mapsto acbca \\ b\mapsto ba \\ c\mapsto cc \end{cases}$$
We compute the incidence matrix and both, the suffix and the prefix matrix for $\sigma$:
$$
M_{\sigma} = 
\left[
\begin{array}{ccc}
2 & 1 & 0 \\
1 & 1 & 0  \\
2 & 0 & 2
\end{array} 
\right],
\;\;
S_{\sigma} = 
\left[
\begin{array}{ccc}
1 & 1 & 0 \\
0 & 0 & 0 \\
0 & 0 & 1
\end{array} 
\right],
\;\;
P_{\sigma} = 
\left[
\begin{array}{ccc}
1 & 0 & 0 \\
0 & 1 & 0 \\
0 & 0 & 1
\end{array}
\right].
$$
The occurrences of words of length 2 in the 
generator images $\sigma(a), \sigma(b)$ and $\sigma(c)$, 
which serve to determine the columns of the lower left off-diagonal block of 
the augmented incidence matrix, are given respectively 
(with indices in lexicographic order) by:
$$
^t(0, 0, 1 \,\vdots\, 0, 0, 1 \,\vdots\, 1, 1, 0),\, ^t(0, 0, 0 \,\vdots\, 1, 0, 0 \,\vdots\, 0, 0, 0),\, ^t(0, 0, 0 \,\vdots\, 0, 0, 0 \,\vdots\, 0, 0, 1)$$
We have now all ingredients to compute the
augmented incidence matrix $M_\sigma^+$, which is of size $12 \times 12$. Recall that the 
upper diagonal block (of size $3 \times 3$) is equal to $M_\sigma$, and the
lower block (of size $9\times 9$) is equal to the Kronecker product $S_\sigma \otimes P_\sigma$.
$$M_\sigma^+ = \left[
\begin{array}{ccc|ccc:ccc:ccc}
2&1&0 & 0&0&0 & 0&0&0 & 0&0&0\\
1&1&0 & 0&0&0 & 0&0&0 & 0&0&0\\
2&0&2 & 0&0&0 & 0&0&0 & 0&0&0\\
\hline
0&0&0 & 1&0&0 & 1&0&0 & 0&0&0\\ 
0&0&0 & 0&1&0 & 0&1&0 & 0&0&0\\ 
1&0&0 & 0&0&1 & 0&0&1 & 0&0&0\\
\hdashline
0&1&0 & 0&0&0 & 0&0&0 & 0&0&0\\
0&0&0 & 0&0&0 & 0&0&0 & 0&0&0\\  
1&0&0 & 0&0&0 & 0&0&0 & 0&0&0\\
\hdashline
1&0&0 & 0&0&0 & 0&0&0 & 1&0&0\\
1&0&0 & 0&0&0 & 0&0&0 & 0&1&0\\
0&0&1 & 0&0&0 & 0&0&0 & 0&0&1\\
\end{array}
\right]$$

We now observe that the incidence matrix $M_\sigma$ itself has two primitive diagonal blocks, with eigenvalue 
$\lambda_2 = 2$ for the bottom block of $M_\sigma$, and $\lambda_1 = 1 + \phi$ for the top one
(with $\phi = \frac{1 + \sqrt{5}}{2}$ as before). 
Since $1 + \phi > 2$ holds, 
the case (2) of 
Corollary \ref{2-strata-case} applies, giving rise to two non-negative eigenvectors $\vec v_1$ for $\lambda_1$ and $\vec v_2$ for $\lambda_2$,  which determine invariant ergodic probability measures $\mu_1$ and $\mu_2$ respectively on the substitution subshift $X_\sigma$. As explained in Step 3 of section \ref{cylinder-determination}, each $\vec v_i$ gives rise to a non-negative eigenvector $\vec v_i^+$ of the augmented matrix $M_\sigma^+$, which we can use to determine the measure of specific cylinders $[w]$ via Proposition \ref{cylinder-formula2}. But first we consider the words of length 1 and 2, as for their cylinders the measure is read off directly from $\vec v_i$ and $\vec v_i^+$ respectively, using Remark \ref{rem-cylindres2}:

$\bullet$ For the eigenvalue $\lambda_2 = 2$ we compute 
the eigenvector $\vec v_2$ and we obtain for the 
words of length~$1$:

$$\mu_2([a])=\mu_2([b])=0, \,\,\, \mu_2([c])=1$$
The computation of the augmented eigenvector $\vec v_2^+$ shows that every cylinder of size $2$ has measure $0$, except $[cc]$. Thus the measure $\mu_2$ is atomic: It only gives positive measure to the biinfinite periodic word 
$\ldots ccc \ldots$.

$\bullet$ For the eigenvalue $\lambda_1 = 1+\varphi = \phi^2$ we compute 
the eigenvector 
$$\vec v_1^+ = \frac{1}{3} \, {}^t(\phi-1, 2-\phi, 2 \mid 5\phi-8, 0, 
7 -4\phi \,\vdots\, 5-3\phi,  0 , 2\phi-3 \,\vdots\, 2-\phi , 2-\phi, 2\phi-2)$$ 
which 
defines the second ergodic measure $\mu_1$. We apply Remark \ref{rem-cylindres2} to obtain:

\[\begin{array}{|c|c|c|c|}
\hline
\rule [-0.1cm]{0mm}{0.5cm}
w&a&b&c\\
\hline
\rule [-0.3cm]{0mm}{0.8cm}
\mu_1([w])&\frac{\varphi-1}{3}&\frac{2-\varphi}{3}&\frac{2}{3}\\
\hline
\end{array}\]

\[\begin{array}{|c|c|c|c|c|c|c|c|c|c|}
\hline
\rule [-0.1cm]{0mm}{0.5cm}
w&aa&ab&ac&ba&bb&bc&ca&cb&cc\\
\hline
\rule [-0.3cm]{0mm}{0.8cm}
\mu_1([w])&
\frac{5\varphi-8}{3}&
0&
\frac{7-4\varphi}{3}&
\frac{5-3\varphi}{3}&
0&
\frac{2\varphi-3}{3}&
\frac{2-\varphi}{3}&
\frac{2-\varphi}{3}&
\frac{2\varphi-2}{3}\\

\hline
\end{array}\]

To compute the measure of a cylinder $[w]$ for a word $w$ of length larger than 2, say $w = bcacc$, we first iterate the given substitution to calculate $\sigma^n$ on the generators $a_i$, 
until their images have length $|\sigma^n(a_i)| \geq|w|-1 = 4$:
$$\begin{cases} \sigma^2(a) =  acbcaccbaccacbca \\ 
\sigma^2(b) = baacbca \\ 
\sigma^2(c) = cccc \end{cases}$$
We then compute the corresponding occurrence vector 
$$^t\vec v(bcacc)_2 = (1,0,0\mid 0,0,1\,\vdots\,0,0,1\,\vdots\,0, 0, 0)$$
and now, according to Proposition \ref{cylinder-formula2}, the matrix product
$$\frac{1}{\lambda_i^2}\,^t\vec v(bcacc)_2 \cdot \vec v_i^+ \, .$$
For $i = 2$ this gives of course $\mu_2([bdacc]) = 0$, while for $i=1$ we obtain:
$$\mu_1([bdacc]) = \frac{1}{\phi^4}(\frac{\varphi-1}{3} + \frac{7-4\varphi}{3} + \frac{2\varphi-3}{3}) = \frac{3-\phi}{\phi^4}$$

\medskip
\subsection{Example from Bezuglyi, Kwiatkowski, Medynets and Solomyak}
\label{third-example}

${}^{}$

This example has already been treated in \cite{BKMS}, Example 5.8 (up to a permutation of the letters); we present here the alternative treatment by our methods.
\[
\sigma:  \, \begin{cases} a\mapsto baaad\\ b\mapsto bc\\ c\mapsto cb\\ d\mapsto de\\ e\mapsto ed\end{cases}\]
The augmented incidence matrix $M_\sigma^+$ has size $30$. Its upper diagonal block is equal to the incidence matrix $M_\sigma$, while the lower one is given by $S\otimes P$. We compute:
\[ M_\sigma=\begin{pmatrix}
3&0&0&0&0\\ 
1&1&1&0&0\\ 
0&1&1&0&0\\ 
1&0&0&1&1\\ 
0&0&0&1&1
\end{pmatrix}, \,\,\,
S=\begin{pmatrix}
0&0&0&0&0\\ 
0&0&1&0&0\\ 
0&1&0&0&0\\ 
1&0&0&0&1\\ 
0&0&0&1&0
\end{pmatrix}, \,\,\,
P=\begin{pmatrix}
0&0&0&0&0\\ 
1&1&0&0&0\\ 
0&0&1&0&0\\ 
0&0&0&1&0\\ 
0&0&0&0&1
\end{pmatrix} 
\]
Next we list the occurrences of words of length 2 in the images of the generators, which serve to determine the columns of the lower left off-diagonal block of 
$M_\sigma^+$. As before, these columns are written 
in transposed form, with $\cal A = \{a, b, c, d, e\}$ and indices ordered lexicographically:
$$(|\sigma(a)|_{xy})_{x, y \in \cal A} = (2,0,0,1,0\,\vdots\, 1,0,0,0,0\,\vdots\,0,0,0,0,0\,\vdots\,0,0,0,0,0\,\vdots\,0,0,0,0,0)
$$
$$(|\sigma(b)|_{xy})_{x, y \in \cal A} = (0,0,0,0,0\,\vdots\, 0,0,1,0,0\,\vdots\,0,0,0,0,0\,\vdots\,0,0,0,0,0\,\vdots\,0,0,0,0,0)
$$
$$(|\sigma(c)|_{xy})_{x, y \in \cal A} = (0,0,0,0,0\,\vdots\, 0,0,0,0,0\,\vdots\,0,1,0,0,0\,\vdots\,0,0,0,0,0\,\vdots\,0,0,0,0,0)
$$
$$(|\sigma(d)|_{xy})_{x, y \in \cal A} = (0,0,0,0,0\,\vdots\, 0,0,0,0,0\,\vdots\,0,0,0,0,0\,\vdots\,0,0,0,0,1\,\vdots\,0,0,0,0,0)
$$
$$(|\sigma(e)|_{xy})_{x, y \in \cal A} = (0,0,0,0,0\,\vdots\, 0,0,0,0,0\,\vdots\,0,0,0,0,0\,\vdots\,0,0,0,0,0\,\vdots\,0,0,0,1,0)
$$

\smallskip

We now observe that all three diagonal blocks of $M_\sigma$ are distinguished, so that for each of them there is a non-negative eigenvector of $M_\sigma^+$ with eigenvalue $> 1$, giving rise to 3 distinct ergodic probability measures $\mu_{\text{full}}, \mu_{b,c}$ and $\mu_{d, e}$ on the substitution subshift $X_\sigma$.

\medskip
$\bullet$ 
For each of {\em the two lower diagonal blocks} the substitution $\sigma$ defines a ``sub-substitution'' $\sigma_{b, c}$ and $\sigma_{d,e}$, since both, the submonoids generated by $b$ and $c$ as well as that for $d$ and $e$, are $\sigma$-invariant. Both of these sub-substitutions are (up to renaming of the generators) equal to the Thue-Morse substitution $\sigma_{\rm TM}$ from section \ref{cylinder-determination}. They define minimal sub-subshifts $X_{b, c} \subset X_\sigma$ and $X_{d, e} \subset X_\sigma$, which are precisely the support of the measures $\mu_{inc}$ and $\mu_{d,e}$ respectively. Since the Thue-Morse substitutions has already been treated in detail in section \ref{cylinder-determination}, we will skip here the corresponding computations.

\medskip

$\bullet$ 
The {\em top diagonal block} (of size $1 \times 1$) of $M_\sigma$ has eigenvalue $3$ and determines a non-negative eigenvector $\vec v^+$ of $M_\sigma^+$, which is chosen here as to have integer coefficients. In order to determine from $\vec v^+$ via Remark \ref{rem-cylindres2} the $\mu_{\text{full}}$-measures for the cylinders of size 1 and 2 we then rescale the values of the coordinates (by the factor $\frac{1}{36}$) to get total measure $1$.

The first 5 coordinates of $\vec v^+$ define a ``subvector'' $\vec v = (v_x)_{x \in \cal A}$ which satisfies 
$M_\sigma \vec v=3 \vec v$.
We compute $\vec v = {}^t(12, 8, 4, 8, 4)$, which
gives:
$$
\mu_{\text{full}}([a]) = \frac{1}{3}, \,  \mu_{\text{full}}([b]) \frac{2}{9},\, \mu_{\text{full}}([c])= \frac{1}{9}, \, \mu_{\text{full}}([d])= \frac{2}{9}, \, \mu_{\text{full}}([e])=\frac{1}{9}
$$

We then use the above described coefficients, for the lower diagonal 
and the lower left off-diagonal block of the matrix $M_\sigma^+$, to compute 
the last 25 coordinates $v_{xy}$ of $\vec v^+$, with $x, y \in \cal A$:
\[ \begin{array}{|c|c|c|c|c|c|c|c|c|c|c|c|c|}
\hline
\rule [-0.2cm]{0mm}{0.3cm}
v_{aa}&v_{ab}&v_{ac}&v_{ad}&v_{ae}&&v_{ba}&v_{bb}&v_{bc}&v_{bd}&v_{be}\\
\hline
\rule [0cm]{0mm}{0.4cm}
8&0&0&4&0&&4&1&3&0&0\\
\hline
\end{array}\]
\[
 \begin{array}{|c|c|c|c|c|c|c|c|c|c|c|c|c|}
\hline
\rule [-0.2cm]{0mm}{0.3cm}
v_{ca}&v_{cb}&v_{cc}&v_{cd}&v_{ce}&&v_{da}&v_{db}&v_{dc}&v_{dd}&v_{de}\\
\hline
\rule [0cm]{0mm}{0.4cm}
0&3&1&0&0&&0&3&0&2&3\\
\hline
\end{array}
\]

\[ \begin{array}{|c|c|c|c|c|}
\hline
\rule [-0.2cm]{0mm}{0.3cm}
v_{ea}&v_{eb}&v_{ec}&v_{ed}&v_{ee}\\
\hline
\rule [0cm]{0mm}{0.4cm}
0&1&0&2&1\\

\hline
\end{array}\]

Hence all cylinders of size 2 have $\mu_{\text{full}}$-measure 0, except for:
$$
\mu_{\text{full}}([aa]) = \frac{2}{9}, \, \mu_{\text{full}}([ad]) = \frac{1}{9}, \, \mu_{\text{full}}([ba]) = \frac{1}{9}, \, \mu_{\text{full}}([bb]) = \frac{1}{36}, \,  \mu_{\text{full}}([bc]) = \frac{1}{12}, 
$$
$$
\mu_{\text{full}}([cb]) = \frac{1}{12}, \, \mu_{\text{full}}([cc]) = \frac{1}{36}, \, \mu_{\text{full}}([db]) = \frac{1}{12},\, \mu_{\text{full}}([dd]) = \frac{1}{18},
$$
$$
\mu_{\text{full}}([de]) = \frac{1}{12}, \, \mu_{\text{full}}([eb]) = \frac{1}{36}, \, \mu_{\text{full}}([ed]) = \frac{1}{18}, \, \mu_{\text{full}}([ee]) = \frac{1}{36}
$$

We observe that 
%(contrary to what has been claimed in Example 5.8 of \cite{BKMS}) 
$\mu_{\text{full}}$ has indeed full support on $X_\sigma$, as follows in more generality from the fact that the eigenvector $\vec v$ of $M_\sigma$ is positive.

Since the $\sigma$-image of every 
generator $a_i$ has length $|a_i| \geq 2$, 
we can use the above values also for a direct evaluation via Proposition \ref{cylinder-formula2} of any cylinder $[w]$ of size $|w| = 3$, through a quick calculation of the occurrence vector $\vec v(w)_0$. For example, this occurence vector for $w = edb$ has coefficients equal to 1 only in the coordinates $ea$ and $eb$, giving 
$$\mu_{\text{full}}([edb]) = \mu_{\text{full}}([ea]) + \mu_{\text{full}}([eb]) = \frac{1}{36}\, ,$$
while for $w = cba$ we obtain
$$\mu_{\text{full}}([cba]) = \mu_{\text{full}}([ba]) = \frac{1}{9} \, .$$

%%%%%%%%
\medskip
\subsection{A family of examples with varying number of ergodic measures}
\label{second-example}

${}^{}$

For any integer $k \geq 1$ we consider the substitution $\sigma_k$ defined by:
$$\tau:\begin{cases}c\mapsto cd\\ d\mapsto c\end{cases}, \qquad \sigma_k:\begin{cases}a\mapsto bacaab\\ b\mapsto aba\\ c\mapsto \tau^k(c)\\ d\mapsto \tau^k(d)\end{cases}$$
We compute the incidence matrix for $\sigma_k$:
$$
M_{\sigma_k} = 
\left[
\begin{array}{cccc}
3& 2 & 0 & 0 \\
2 & 1 & 0 & 0  \\
1 & 0 &  a_k& b_k  \\
0 & 0 & c_k & d_k 
\end{array} 
\right]
\quad \text{with} \quad
\left[
\begin{array}{cc}
a_k & b_k  \\
c_k&d_k
\end{array} 
\right]
=\left[
\begin{array}{cc}
1 & 1  \\
1&0
\end{array} 
\right]^k = M_\tau^k
$$
We note that for any $k \geq 1$ the Corollary \ref{2-strata-case} applies, and since $\tau$ is up to a change of generators equal to the Fibonacci substitution $\sigma_{\rm Fib}$ treated extensively in section \ref{cylinder-determination}, we know already that the primitive bottom diagonal block $M_\tau^k$ of $M_{\sigma_k}$ has PF-eigenvalue $\phi^k$. Its corresponding eigenvector determines an invariant probability measure which is supported only on the bottom stratum of $\sigma_k$, defined by $c$ and $d$, and the cylinder values are precisely those computed in section \ref{cylinder-determination} for the Fibonacci substitution.

The top diagonal block of $M_{\sigma_k}$ is independent of $k$ and turns out to be actually equal to $M_\tau^3$: it thus has PF-eigenvalue $\phi^3$. Hence for $k \geq 3$ case (1) of Corollary \ref{2-strata-case} holds, so that the above described invariant ``Fibonacci'' measure supported on the bottom stratum is the only invariant measure on the substitution subshift $X_{\sigma_k}$.

For $k=1$ and $k=2$, however, we find ourselves in case (2) of Corollary \ref{2-strata-case}, and hence in both cases there is a second invariant probability measure $\mu_k$ with full support, which we will now consider in detail:

\medskip
\noindent
{\bf The case $k=1$:} 
We first compute an integer PF-eigenvector 
$^t(v_a, v_b, v_c, v_d)$ 
for the eigenvalue $\phi^3$ of the above given (non-augmented) incidence matrix $M_{\sigma_1}$, 
where we note that $\varphi^3=2\varphi+1$:
\[\begin{array}{|c|c|c|c|}
\hline
\rule [-0.2cm]{0mm}{0.3cm}
v_a&v_b&v_c&v_d\\
\hline
\rule [-0.2cm]{0mm}{0.65cm}
3\varphi^3&3\varphi^2& \varphi^3&1\\
\hline
\end{array}\]
This gives directly the measure of the cylinders of size 1, through the normalization $\mu_1([x]) =\frac{1}{11\varphi+8} v_x$ for any $x \in \cal A = \{a, b, c, d\}$.

For the cylinders of size 2 
we compute (as in the previous examples) the augmented incidence matrix $M_{\sigma_1}^+$ of $\sigma_1$, which has size $20$. We note:
$$
S_{\sigma_1} = 
\left[
\begin{array}{cccc}
0 & 1 & 0 & 0 \\
1 & 0 & 0 & 0 \\
0 & 0 & 0 & 1 \\
0 & 0 & 1 & 0
\end{array} 
\right],
\;\;
P_{\sigma_1} = 
\left[
\begin{array}{cccc}
0 & 1 & 0 & 0 \\
1 & 0 & 0 & 0 \\
0 & 0 & 1 & 1 \\
0 & 0 & 0 & 0
\end{array}
\right].
$$
Next we list the occurrences of words of length 2 in the images of the generators, which serve to determine the columns of the lower left off-diagonal block of 
$M_\sigma^+$. These columns are written as in the previous examples in transposed form, with 
indices ordered lexicographically:
$$
(|\sigma_1(a)|_{xy})_{x, y \in \cal A} = 
(1,1,1,0 \,\vdots\, 1,0,0,0\,\vdots\,1,0,0,0\,\vdots\,0,0,0,0)
$$
$$
(|\sigma_1(b)|_{xy})_{x, y \in \cal A} = 
(0,1,0,0\,\vdots\,1,0,0,0\,\vdots\,0,0,0,0\,\vdots\,0,0,0,0)
$$
$$
(|\sigma_1(c)|_{xy})_{x, y \in \cal A} = 
(0,0,0,0\,\vdots\,0,0,0,0\,\vdots\,0,0,0,1\,\vdots\,0,0,0,0)
$$
$$
(|\sigma_1(d)|_{xy})_{x, y \in \cal A} = 
(0,0,0,0\,\vdots\,0,0,0,0\,\vdots\,0,0,0,0\,\vdots\,0,0,0,0)
$$
We now compute a right PF-eigenvector of $M_{\sigma_1}^+$ associated to $\varphi^3$, with coefficients $v_w$ for any $|w| \leq 2$ that agree for $|w| = 1$ with the above coefficients $v_x$ for the eigenvector of $M_{\sigma_1}$.
From the above specified coefficients of $M_{\sigma_1}^+$ we derive
through a minor computational effort:
$$\begin{cases}
v_a+v_{bb}=\varphi^3 v_{aa}\\
v_a+v_b+v_{ba}=\varphi^3 v_{ab}\\
v_a+v_{bc}+v_{bd}=\varphi^3 v_{ac}\\
0=\varphi^3 v_{ad}\end{cases}
\begin{cases}
v_a+v_b+v_{ab}=\varphi^3 v_{ba}\\
v_{aa}=\varphi^3 v_{bb}\\
v_{ac}+v_{ad}=\varphi^3 v_{bc}\\
0=\varphi^3 v_{bd}\end{cases}$$
$$
\begin{cases}
v_a + v_{db}=\varphi^3 v_{ca}\\
v_{da}=\varphi^3 v_{cb}\\
v_{dc}+v_{dd}=\varphi^3 v_{cc}\\
v_c=\varphi^3 v_{cd}\end{cases}
\begin{cases}
v_{cb}=\varphi^3 v_{da}\\
v_{ca}=\varphi^3 v_{db}\\
v_{cc}+v_{cd}=\varphi^3 v_{dc}\\
0=\varphi^3 v_{dd}\end{cases}
$$
This gives:
\[
\begin{array}{|c|c|c|c|c|c|c|c|c|c|c|}
\hline
\rule [-0.2cm]{0mm}{0.3cm}
v_{aa}&v_{ab}&v_{ac}&v_{ad}&&v_{ba}&v_{bb}&v_{bc}&v_{bd}\\
\hline
\rule [-0.35cm]{0mm}{0.85cm}
\frac{\varphi^3 v_a}{\phi^6 -1}&\frac{v_a + v_b}{\phi^3 -1}&\frac{\varphi^3 v_a}{\phi^6 -1}&0&&\frac{v_a + v_b}{\phi^3 -1}&\frac{v_a}{\phi^6 -1}&\frac{v_a}{\phi^6 -1}&0\\
\hline
\end{array}\]
\[
\begin{array}{|c|c|c|c|c|c|c|c|c|c|c|}
\hline
\rule [-0.2cm]{0mm}{0.3cm}
v_{ca}&v_{cb}&v_{cc}&v_{cd}&&v_{da}&v_{db}&v_{dc}&v_{dd}\\
\hline
\rule [-0.35cm]{0mm}{0.85cm}
\frac{\phi^3 v_a}{\phi^6 -1}&0&\frac{1}{\phi^6 -1}&1&&0&\frac{v_a}{\phi^6 -1}&\frac{v_c}{\phi^6 -1}&0\\
\hline
\end{array}
\]
As above, we obtain the measure of the cylinders of size $|w|= 2$ through the normalization $\mu_1([w]) =\frac{1}{11\varphi+8} v_w$.

\medskip
\bigskip
\noindent
{\bf The case $k=2$:} 
We proceed precisely as in the case $k = 1$ to compute from $M_{\sigma_2}$ an integer eigenvector for the eigenvalue $\phi^3$:
\[\begin{array}{|c|c|c|c|}
\hline
\rule [-0.2cm]{0mm}{0.3cm}
v_a&v_b&v_c&v_d\\
\hline
\rule [-0.2cm]{0mm}{0.65cm}
5\varphi+2&2\varphi+3&2\varphi+2&\varphi\\
\hline
\end{array}\]
We next determine 
$$
S_{\sigma_2} = 
\left[
\begin{array}{cccc}
0 & 1 & 0 & 0 \\
1 & 0 & 0 & 0 \\
0 & 0 & 1 & 0 \\
0 & 0 & 0 & 1
\end{array} 
\right],
\;\;
P_{\sigma_2} = 
\left[
\begin{array}{cccc}
0 & 1 & 0 & 0 \\
1 & 0 & 0 & 0 \\
0 & 0 & 1 & 1 \\
0 & 0 & 0 & 0
\end{array}
\right].
$$
as well as: 
$$
(|\sigma_2(a)|_{xy})_{x, y \in \cal A} = 
(1,1,1,0\,\vdots\,1,0,0,0\,\vdots\,1,0,0,0\,\vdots\,0,0,0,0)
$$
$$
(|\sigma_2(b)|_{xy})_{x, y \in \cal A} = 
(0,1,0,0\,\vdots\,1,0,0,0\,\vdots\,0,0,0,0\,\vdots\,0,0,0,0)
$$
$$
(|\sigma_2(c)|_{xy})_{x, y \in \cal A} = 
(0,0,0,0\,\vdots\,0,0,0,0\,\vdots\,0,0,0,1\,\vdots\,0,0,1,0)
$$
$$
(|\sigma_2(d)|_{xy})_{x, y \in \cal A} = 
(0,0,0,0\,\vdots\,0,0,0,0\,\vdots\,0,0,0,1\,\vdots\,0,0,0,0)
$$
We thus compute
$$\begin{cases}
v_a+v_{bb}=\varphi^3 v_{aa}\\
v_a+v_b+v_{ba}=\varphi^3 v_{ab}\\
v_a+v_{bc}+v_{bd}=\varphi^3 v_{ac}\\
0=\varphi^3 v_{ad}\end{cases}
\begin{cases}
v_a+v_b+v_{ab}=\varphi^3 v_{ba}\\
v_{aa}=\varphi^3 v_{bb}\\
v_{ac}+v_{ad}=\varphi^3 v_{bc}\\
0=\varphi^3 v_{bd}\end{cases}$$
$$
\begin{cases}
v_a+v_{cb}=\varphi^3 v_{ca}\\
v_{ca}=\varphi^3 v_{cb}\\
v_{cc}+v_{cd}=\varphi^3 v_{cc}\\
v_c+v_d=\varphi^3 v_{cd}\end{cases}
\begin{cases}
v_{db}=\varphi^3 v_{da}\\
v_{da}=\varphi^3 v_{db}\\
v_c+v_{dc}+v_{dd}=\varphi^3 v_{dc}\\
0=\varphi^3 v_{dd}\end{cases}
$$
and obtain:
\[
\begin{array}{|c|c|c|c|c|c|c|c|c|c|c|}
\hline
\rule [-0.2cm]{0mm}{0.3cm}
v_{aa}&v_{ab}&v_{ac}&v_{ad}&&v_{ba}&v_{bb}&v_{bc}&v_{bd}\\
\hline
\rule [-0.35cm]{0mm}{0.85cm}
\frac{\varphi^3 v_a}{\phi^6 -1}&\frac{v_a + v_b}{\phi^3 -1}&\frac{\varphi^3 v_a}{\phi^6 -1}&0&&\frac{v_a + v_b}{\phi^3 -1}&\frac{v_a}{\phi^6 -1}&\frac{v_a}{\phi^6 -1}&0\\
\hline
\end{array}\]

\[
\begin{array}{|c|c|c|c|c|c|c|c|c|c|c|}
\hline
\rule [-0.2cm]{0mm}{0.3cm}
v_{ca}&v_{cb}&v_{cc}&v_{cd}&&v_{da}&v_{db}&v_{dc}&v_{dd}\\
\hline
\rule [-0.35cm]{0mm}{0.85cm}
\frac{v_a\varphi^3}{\varphi^6-1}&\frac{v_a}{\varphi^6-1}&\frac{v_c+v_d}{\varphi^3(\varphi^3-1)}&\frac{v_c+v_d}{\varphi^3}&&0&0&\frac{v_c}{\phi^3-1}&0\\
\hline
\end{array}
\]
As for $k=1$ we obtain the measure of any cylinder of size $|w|\leq 2$ from the coefficients $v_w$ through a normalization given by:
$$\mu_2([w]) =\frac{v_w}{10\varphi+7}$$

%%%%%%%%%%%%%%%%%%%%%%
%%%%
%%%%%%%%%%%%
%%%%
%%%%%%%%%%%%%%%%%%%%%%

%\bibliographystyle{abbrv}
%\bibliography{biblio-BHL2}

\end{document}